\documentclass[11pt]{article} 

\usepackage[utf8]{inputenc} 

\usepackage[margin=1in]{geometry} 
\usepackage{graphicx} 

\usepackage{booktabs} 
\usepackage{array} 
\usepackage{paralist} 
\usepackage{verbatim} 
\usepackage{subfig} 
\usepackage{mathrsfs}
\usepackage{amssymb}
\usepackage{amsthm}
\usepackage{amsmath,amsfonts,amssymb,esint}
\usepackage{graphics}
\usepackage{enumerate}
\usepackage{mathtools}
\usepackage{xfrac}
\usepackage{bbm}

\usepackage{fancyhdr} 
\usepackage[nottoc,notlof,notlot]{tocbibind} 
\usepackage[titles,subfigure]{tocloft} 

\numberwithin{equation}{section}

\newtheorem{theorem}{Theorem}[section]
\newtheorem{corollary}[theorem]{Corollary}
\newtheorem{proposition}[theorem]{Proposition}
\newtheorem{lemma}[theorem]{Lemma}
\newtheorem{definition}[theorem]{Definition}
\theoremstyle{definition}
\newtheorem{remark}[theorem]{Remark}

\newcommand{\norm}[1]{\left\|#1\right\|}
\newcommand{\abs}[1]{\left|#1\right|}

\newcommand*{\supp}{\ensuremath{\mathrm{supp\,}}}
\newcommand*{\Id}{\ensuremath{\mathrm{Id}}}
\renewcommand*{\div}{\ensuremath{\mathrm{div\,}}}

\newcommand*{\N}{\ensuremath{\mathbb{N}}}

\newcommand*{\T}{\ensuremath{\mathbb{T}}}
\newcommand*{\Z}{\ensuremath{\mathbb{Z}}}
\newcommand*{\R}{\ensuremath{\mathbb{R}}}
\newcommand*{\CC}{\ensuremath{\mathbb{C}}}
\newcommand*{\tr}{\ensuremath{\mathrm{tr\,}}}
\newcommand{\eps}{\varepsilon}

\newcommand{\RR}{\mathring R}
\newcommand{\RSZ}{\mathcal R}
\renewcommand*{\Re}{\ensuremath{\mathrm{Re\,}}}
\renewcommand*{\tilde}{\widetilde}

\newcommand*{\curl}{\ensuremath{\mathrm{curl\,}}}

\newcommand{\WW}{\ensuremath{\mathbb{W}}}
\newcommand{\Proj}{\ensuremath{\mathbb{P}}}
\newcommand{\les}{\lesssim}

\usepackage[usenames,dvipsnames]{color}
\usepackage[colorlinks=true, pdfstartview=FitV, linkcolor=blue, citecolor=blue, urlcolor=blue]{hyperref}

\title{Nonuniqueness of weak solutions to the Navier-Stokes equation}
\author{Tristan Buckmaster
\thanks{Department of Mathematics, Princeton University.
{\footnotesize \href{mailto:buckmaster@math.princeton.edu}{buckmaster@math.princeton.edu}.}
}
\and 
Vlad Vicol
\thanks{Department of Mathematics, Princeton University. 
{\footnotesize \href{mailto:vvicol@math.princeton.edu}{vvicol@math.princeton.edu}.}
}
}

\begin{document}

\maketitle

\begin{abstract}
For initial datum of finite kinetic energy, Leray has proven in 1934 that there exists at least one global in time finite energy weak solution of the 3D Navier-Stokes equations. In this paper we prove that weak solutions of the 3D Navier-Stokes equations are not unique in the class of weak solutions with finite kinetic energy.  Moreover, we prove that H\"older continuous dissipative weak solutions of the 3D Euler equations may be obtained as a strong vanishing viscosity limit of a sequence of finite energy weak solutions of the 3D Navier-Stokes equations.
\end{abstract}


\section{Introduction}
In this paper we consider the 3D incompressible Navier-Stokes equation
\begin{subequations}
\label{eq:NSE}
\begin{align}
\partial_t v + \div (v \otimes v) + \nabla p -  {\nu}\Delta v &=0
\label{eq:NSE:1}\\
\div v &= 0
\label{eq:NSE:2}
\end{align}
\end{subequations}
posed on $\T^3\times \R$, with periodic boundary conditions in $x\in\mathbb T^3=\mathbb R^3 / 2\pi\mathbb Z^3$. We consider solutions normalized to have zero spatial mean, i.e., $\int_{\T^3} v(x,t) dx = 0$. The constant $\nu \in (0,1] $ is the kinematic viscosity. We define weak solutions to the Navier-Stokes equations~\cite[Definition~1]{Serrin63},~\cite[pp.~226]{FabesJonesRiviere72}:
\begin{definition}
\label{def:weak}
We say $v \in C^0(\R;L^2(\T^3))$ is a weak solution of \eqref{eq:NSE} if for any $t \in \R$ the vector field $v(\cdot,t)$ is weakly divergence free, has zero mean, and \eqref{eq:NSE:1} is satisfied  in $\mathcal D'(\T^3 \times \R)$, i.e.,
\begin{align*}
\int_{\R} \! \int_{\T^3} v \cdot (\partial_t \varphi + (v \cdot \nabla) \varphi + \nu \Delta \varphi ) dx dt = 0
\end{align*}
holds for any test function $\varphi \in C_0^\infty(\T^3 \times \R)$ such that $ \varphi(\cdot,t) $ is divergence-free for all $t$.
\end{definition}
As a direct result of the work of Fabes-Jones-Riviere~\cite{FabesJonesRiviere72}, since the weak solutions defined above lie in $C^0(\R;L^2(\T^3))$, they are in fact solutions of the integral form of the Navier-Stokes equations
\begin{align}
v(\cdot,t) = e^{\nu t  \Delta} v(\cdot,0) + \int_0^t  e^{\nu (t-s)\Delta}\Proj \div(v(\cdot,s)\otimes v(\cdot,s)) ds \, ,
\label{eq:mild}
\end{align}
and are sometimes called {\em mild} or \emph{Oseen} solutions (cf.~\cite{FabesJonesRiviere72} and~\cite[Definition 6.5]{Lemarie16}).  Here $\Proj$ is the Leray projector and $e^{t\Delta}$ denotes convolution with the heat kernel.

\subsection{Previous works}
In~\cite{Leray34}, Leray considered the Cauchy-problem for \eqref{eq:NSE} for initial datum of finite kinetic energy, $v_0 \in L^2$. Leray proved that for any such datum, there exists a global in time weak solution $v \in L^\infty_t L^2_x$, which additionally has the regularity $L^2_t \dot{H}^1_x$, and obeys the energy inequality 
$
\norm{v(t)}_{L^2}^2 + 2 \nu \int_0^t \norm{\nabla u(s)}_{L^2}^2 ds \leq \norm{v_0}_{L^2}^2
$.
Hopf~\cite{Hopf51} established a similar result for the equations posed in a smooth bounded domain, with Dirichlet boundary conditions. 
To date, the question of uniqueness of Leray-Hopf weak solutions for the 3D Navier-Stokes equations remains however open.

 Based on the natural scaling of the equations $v(x,t) \mapsto v_\lambda(x,t) = \lambda v(\lambda x,\lambda^2 t)$, a number of partial regularity results have been established~\cite{Scheffer76,CaffarelliKohnNirenberg82,Lin98,LadyzhenskayaSeregin99,Vasseur07,Kukavica08b}; the local existence for the Cauchy problem has been proven in scaling-invariant spaces~\cite{Kato84,KochTataru01,JiaSverak14}; and conditional regularity has been established under geometric structure assumptions~\cite{ConstantinFefferman93} or assuming a signed pressure~\cite{SereginSverak02}. The conditional regularity and weak-strong uniqueness results known under the umbrella of Ladyzhenskaya-Prodi-Serrin conditions~\cite{KiselevLadyzhenskaya57,Prodi59,Serrin62},  state that if a Leray-Hopf weak solution also lies in $L^p_t L^q_x$, with $\sfrac 2p + \sfrac 3q \leq 1$, then the solution is unique and smooth in positive time. These conditions and their generalizations have culminated with the work of Escauriaza-Seregin-\v Sver\'ak~\cite{EscauriazaSerginSverak03} who proved the $L^\infty_t L^3_x$ endpoint. 
The uniqueness of mild/Oseen solutions is also known under~the Ladyzhenskaya-Prodi-Serrin conditions, cf.~\cite{FabesJonesRiviere72} for $p>3$, and~\cite{FurioliLemarieRieussetTerraneo00,LionsMasmoudi01,LemarieRieusset02,Kukavica06b} for $p=3$.
 Note that the regularity of Leray-Hopf weak solutions, or of bounded energy weak solutions, is consistent with the scaling $\sfrac 2p + \sfrac 3q = \sfrac 32$. In contrast, the additional regularity required to ensure that the energy equality holds in the Navier-Stokes equations is consistent with $\sfrac 24 + \sfrac 34 = \sfrac 54$ for $p=q=4$~\cite{Shinbrot74,Kukavica06}. See~\cite{ConstantinFoias88,Temam01,LemarieRieusset02,RSR16,Lemarie16} for surveys of results on the Navier-Stokes equations.
 
The gap between the scaling of the kinetic energy and the natural scaling of the equations leaves open the possibility of nonuniqueness of weak solutions to \eqref{eq:NSE}. In~\cite{JiaSverak14,JiaSverak15} Jia-\v Sver\'ak proved that non-uniqueness of Leray-Hopf weak solutions in the regularity class $L^\infty_t L^{3,\infty}_x$ holds if a certain spectral assumption holds for a linearized Navier-Stokes operator. While a rigorous proof of this spectral condition remains open, very recently Guillod-\v Sver\'ak~\cite{GuillodSverak17} have provided compelling numerical evidence of it, using a scenario related to the example of Ladyzhenskaya~\cite{Ladyzhenskaya67}. Thus, the works~\cite{JiaSverak15,GuillodSverak17} strongly suggest that the Ladyzhenskaya-Prodi-Serrin regularity criteria are sharp.
 
\subsection{Main results}
In this paper we prove that weak solutions to \eqref{eq:NSE} (in the sense of Definition~\ref{def:weak}) are not unique within the class of weak solutions with bounded kinetic energy. We establish the stronger result\footnote{We denote by $H^\beta$ the $L^2$-based Sobolev space with regularity index $\beta$.  Clearly $C^0_t H^\beta_x \subset C^0_t L^2_x$.}:
\begin{theorem}[{\bf Nonuniqueness of weak solutions}]
\label{thm:compact:support}
There exists $\beta>0$, such that for any  {nonnegative} smooth function $e(t) \colon [0,T] \to \R_{\geq 0}$, there exists $v \in C^0_t ([0,T]; H^\beta_x(\mathbb T^3))$ a weak solution of the Navier-Stokes equations, such that 
$\int_{\T^3} |v(x,t)|^2 \, dx = e(t)$ for all $t \in [0,T]$. Moreover, the associated vorticity $\nabla \times v$ lies in $C^0_t([0,T];L^1_x(\T^3))$.
\end{theorem}
In particular, the above theorem shows that $v \equiv 0$ is not the only weak solution which vanishes at a time slice, thereby implying the nonuniqueness of weak solutions. Theorem~\ref{thm:compact:support} shows that weak solutions may come to rest in finite time, a question posed by Serrin~\cite[pp.~88]{Serrin63}. Moreover, by considering $e_1(t),e_2(t)> 0 $ which are nonincreasing, 
such that $e_1(t)= e_2(t)$ for $t\in[0,\sfrac{T}{2}]$, and $e_1(T)<e_2(T)$, the construction used to prove Theorem~\ref{thm:compact:support} also proves the nonuniqueness of dissipative weak solutions.  

From the proof of Theorem~\ref{thm:compact:support} it is clear that the constructed weak solutions $v$ also have regularity in time, i.e.\  there exists $\gamma>0$ such that $v \in C^\gamma_t ([0,T]; L^2_x(\mathbb T^3))$.  Thus, $v\otimes v$ lies in $C^{\gamma}_t L^{1}_x \cap C^0_t L^{1+\gamma}_x$, and the fact that  $\nabla v \in C^0_t L^1_x$ follows from \eqref{eq:mild} and the maximal regularity of the heat equation.

We note that while the weak solutions Theorem~\ref{thm:compact:support} may attain any smooth energy profile, at the moment we do not prove that they are Leray-Hopf weak solutions, i.e., they do not  obey the energy inequality or have $L^2_t \dot{H}^{1}_x$ integrability. Moreover, the regularity parameter $\beta>0$ cannot be expected to be too large, since at $\beta = \sfrac 12$ one has weak-strong uniqueness~\cite{ConstantinFoias88}. We expect that the ideas used to prove Theorem~\ref{thm:compact:support} will in the future lead to a proof of nonuniqueness of weak solutions in $C^0_t L^p_x$, for any $2\leq p<3$, and the nonuniqueness of Leray-Hopf weak solutions.

The proof of Theorem~\ref{thm:compact:support} builds on several of the  fundamental ideas pioneered by De Lellis-Sz\'ekelyhidi Jr.~\cite{DLSZ09,DLSZ13}. These ideas were used to tackle the Onsager conjecture for the Euler equation~\cite{Eyink94,ConstantinETiti94,ChCoFrSh2008} (set $\nu = 0$ in \eqref{eq:NSE})  via convex integration methods~\cite{Scheffer93,Shnirelman97,BDLISJ15,DLSZ15,Buckmaster15,BDLSZ16}, leading to the resolution of the conjecture by Isett~\cite{Isett16,Isett17}, using a key ingredient by Daneri and Sz\'ekelyhidi~Jr.~\cite{DSZ17}. The construction of dissipative Euler solutions below the Onsager regularity threshold was proven by authors of this paper jointly with De Lellis and Sz\'ekelyhidi Jr.\ in~\cite{BDLSV17}, building on the ideas in~\cite{DSZ17,Isett16}. In order to treat the dissipative term $-\nu \Delta$, not present in the Euler system, we cannot proceed as in~\cite{BSV16,CDLDR17}, since in these works H\"older continuous weak solutions are constructed, which is possible only by using building blocks which are sparse in the frequency variable and for small fractional powers of the Laplacian.  Instead, the main idea, which is also used in~\cite{BMV17}, is to use building blocks for the convex integration scheme which are ``intermittent''. That is, the building blocks we use are spatially inhomogeneous, and have different scaling in different $L^p$ norms. At high frequency, these building blocks attempt to saturate the Bernstein inequalities from Littlewood-Paley theory.
Since they are built by adding eigenfunctions of $\curl$ in a certain geometric manner, we call these building blocks {\em intermittent Beltrami flows}. In particular, the proof of Theorem~\ref{thm:compact:support} breaks down in $2D$, as is expected, since there are not enough spatial directions to oscillate in. The proof of Theorem~\ref{thm:compact:support} is given in Section~\ref{sec:outline} below.

The idea of using intermittent building blocks can be traced back to classical observations in hydrodynamic turbulence, see for instance~\cite{Frisch95}. Moreover, in view of the aforementioned works on the Onsager conjecture for the Euler equations, we are naturally led to consider the set of accumulation points in the vanishing viscosity limit $\nu \to 0$ of the family of weak solutions to the Navier-Stokes equations which we constructed in Theorem~\ref{thm:compact:support}. We prove in this paper that this set of accumulation points, in the $C^0_t L^2_x$ topology, contains all the H\"older continuous weak solutions of the 3D Euler equations:

\begin{theorem}[{\bf Dissipative Euler solutions arise in the vanishing viscosity limit}]\label{thm:NSE:Euler}
For $\bar \beta >0$ let $u \in C^{\bar \beta}_{t,x}(\T^3 \times [-2T,2T])$ be a zero-mean weak solution of the Euler equations. Then there exists $\beta>0$, a sequence $\nu_n \to 0$, and a uniformly bounded sequence $v^{(\nu_n)} \in C^0_t ([0, T]; H^\beta_x(\mathbb T^3))$ of weak solutions to the Navier-Stokes equations, with $v^{(\nu_n)} \to u$ strongly in $C^0_t([0, T]; L^2_x(\mathbb T^3))$.
\end{theorem}

In particular, Theorem~\ref{thm:NSE:Euler} shows that the nonconservative weak solutions to the Euler equations obtained in~\cite{Isett16,BDLSV17} arise in the vanishing viscosity limit of weak solutions to the Navier-Stokes equations. Thus, being a strong limit of weak solutions to the Navier-Stokes equations, in the sense of Definition~\ref{def:weak}, cannot serve as a selection criterion for weak solutions of the Euler equation. 
Whether similar vanishing viscosity results hold for sequences of Leray-Hopf weak solutions, or for suitable weak solutions of \eqref{eq:NSE}, remains a challenging open problem.
The proof of Theorem~\ref{thm:NSE:Euler} is closely related to that of Theorem~\ref{thm:compact:support}, and is also given in Section~\ref{sec:outline} below.

\section{Outline of the convex integration scheme}
\label{sec:outline}

In this section we sketch the proof of Theorem~\ref{thm:compact:support}. For every integer $q \geq 0$ we will construct a solution $(v_q,p_q, \RR_q)$ to the Navier-Stokes-Reynolds system
\begin{subequations}\label{e:NSE_reynolds}
\begin{align}
\partial_t v_q + \div (v_q\otimes v_q)+\nabla p_q - \nu \Delta v_q &= \div\RR_q\\
\div v_q &= 0\, .
\end{align}
\end{subequations} 
where the Reynolds stress $\RR_q$ is assumed to be a trace-free symmetric matrix.

\subsection{Parameters} 
Throughout the proof we fix a sufficiently large, universal constant $b \in 16\N$, and depending on $b$ we fix a regularity parameter $\beta>0$ such that $\beta b^2\leq 4$ and $\beta b \leq \sfrac{1}{40}$. We remark that it is sufficient to take $b = 2^9$ and $\beta = 2^{-16}$.

The relative size of the approximate solution $v_q$ and the Reynolds stress error $\RR_q$ will be measured in terms of a frequency parameter $\lambda_q$ and an amplitude parameter $\delta_q$ defined as
\begin{align}
\lambda_{q} &=   a^{(b^q)}\notag  \\
\delta_{q} &= \lambda_1^{3\beta} \lambda_{q}^{-2\beta}\notag
\end{align}
for some integer $a  \gg 1$ to be chosen suitably. 

\subsection{Inductive estimates}\label{s:inductive}

By induction, we will assume the following estimates\footnote{Here and throughout the paper we use the notation: $\norm{ f}_{L^p} = \norm{f}_{L^\infty_t L^p_x}$, for $1\leq p \leq \infty$, $\norm{f}_{C^N} = \norm{f}_{L^\infty_t C^N_x} = \sum_{0\leq |\alpha| \leq N} \norm{D^\alpha f}_{L^\infty}$,  $\norm{f}_{C^N_{x,t}} = \sum_{0 \leq n + |\alpha| \leq N} \norm{\partial_t^n D^\alpha f}_{L^\infty}$, and $\norm{f}_{W^{s,p}} = \norm{f}_{L^\infty_t W^{s,p}_x}$, for $s>0$, and $1 \leq p \leq \infty$.}
on the solution of \eqref{e:NSE_reynolds} at level $q$:
\begin{align}
\norm{v_q}_{C^1_{x,t}}&\leq \lambda_{q}^4\label{e:V_ind}\\
\norm{\RR_{q}}_{L^{1}} &\leq \lambda_q^{-\eps_R}\delta_{q+1}
\label{eq:R:q+1:ind}\\
\norm{\RR_{q}}_{C^1_{x,t}} &\leq \lambda_{q}^{10}\label{e:R_ind_C1}
\, .
\end{align}
We additionally assume
\begin{equation}
\label{eq:energy_ind}
0 \leq e(t) - \int_{\mathbb T^3}\abs{v_q}^2~dx\leq  \delta_{q+1}
\end{equation}
and 
\begin{equation}\label{eq:zero_reynolds}
e(t) - \int_{\mathbb T^3}\abs{ v_q(x,t)}^2~dx\leq  \frac{\delta_{q+1}}{100}~\Rightarrow  v_q(\cdot,t) \equiv 0 \quad \mbox{and} \quad \mathring R_q(\cdot,t)\equiv 0 \, .
\end{equation}
for all $t\in [0,T]$.

\subsection{The main proposition and iterative procedure}
In addition to the sufficiently large universal constant $b$, and the sufficiently small regularity parameter $\beta = \beta(b)>0$ fixed earlier, we fix the constant $M_e = \norm{e}_{C^1_t}$. The following iteration lemma states the existence of a solution of \eqref{e:NSE_reynolds} at level $q+1$, which obeys suitable bounds.
\begin{proposition}\label{p:main}
There exists a universal constant $M>0$, a sufficiently small parameter $\eps_R = \eps_R(b,\beta)>0$ and a sufficiently large parameter $a_0 = a_0(b, \beta, \eps_R, M,M_e)>0$
such that for any integer $a \geq a_0$, which is a multiple of the $N_\Lambda$ of Remark~\ref{rem:lcm}, the following holds: Let $(v_q,p_q,\mathring R_q)$ be a triple solving the Navier-Stokes-Reynolds system \eqref{e:NSE_reynolds} in $\mathbb T^3\times [0,T]$ satisfying the inductive estimates \eqref{e:V_ind}--\eqref{eq:zero_reynolds}. Then there exists a second triple  $(v_{q+1},p_{q+1},\mathring R_{q+1})$ solving   \eqref{e:NSE_reynolds} and satisfying the \eqref{e:V_ind}--\eqref{eq:zero_reynolds} with $q$ replaced by $q+1$. In addition we have that
\begin{equation}\label{e:interative_v}
\norm{v_{q+1}-v_q}_{L^{2}}\leq M \delta_{q+1}^{\sfrac12}\,.
\end{equation}
\end{proposition}

The principal new idea in the proof of Proposition~\ref{p:main} is to construct the perturbation $v_{q+1}-v_{q}$ as a sum of terms of the form
\begin{align}\label{e:intro_components}
a_{(\xi)}\mathbb W_{(\xi)}
\end{align}
where $\mathbb W_{(\xi)}$ is an intermittent Beltrami wave (cf.~\eqref{e:Dirichlet_Beltrami_def} below) with frequency support centered at frequency $\xi\lambda_{q+1}$ for $\xi\in\mathbb S^2$. While these intermittent Beltrami waves have similar properties (cf.~Proposition~\ref{p:gamma_def}) to the usual Beltrami flows used in the previous convex integration constructions~\cite{DLSZ13,BDLISJ15,DLSZ15,Buckmaster15,BDLSZ16} for the Euler equations, they are fundamentally different since their $L^1$ norm is much smaller than their $L^2$ norm (cf.~Proposition~\ref{p:W_bounds}). The gain comes from the fact that the Reynolds stress has to be estimated in $L^1$ rather than $L^2$, and that the term $\nu \Delta v$ is linear in $v$. At the technical level, one difference with respect to~\cite{Isett16,BDLSV17} is the usage of very large gaps between consecutive frequency parameters (i.e., $b\gg 1$), which is consistent with a small regularity parameter $\beta$. Next, we show that Proposition~\ref{p:main} implies the main theorems of the paper.

\subsection{Proof of Theorem~\ref{thm:compact:support}}
Choose all the parameters from the statement of Proposition~\ref{p:main}, except for $a$, which we may need to be larger (so that it is still larger than $a_0$).

For $q=0$ we note that the identically zero solution trivially satisfies \eqref{e:NSE_reynolds} with $\RR_0= 0$, and the inductive assumptions \eqref{e:V_ind}, \eqref{eq:R:q+1:ind}, and \eqref{e:R_ind_C1} hold.
Moreover, by taking $a$ sufficiently large such that it is in the range of Proposition \ref{p:main} (i.e. $a \geq a_0$) we may  ensure that
\begin{equation*}
|e(t)| \leq \norm{e}_{C^1_t} = M_e \leq  \frac{\lambda_1^\beta }{100} = \frac{\delta_{1}}{100}.
\end{equation*}
Then the zero solution also satisfies \eqref{eq:energy_ind} and \eqref{eq:zero_reynolds}. 

For $q \geq 1$ we inductively apply Proposition~\ref{p:main}. 
The bound \eqref{e:interative_v} and interpolation implies\footnote{Throughout this paper, we we will write $A \les B$ to denote that there exists a sufficiently large constant $C$, which is independent of $q$, such that $A \leq C B$.}
\begin{align}
\sum_{q=0}^{\infty} \norm{v_{q+1}-v_q}_{H^{\beta'}} 
&\lesssim \; \sum_{q=0}^{\infty} \norm{v_{q+1}-v_q}_{L^{2}}^{1-\beta'}(\norm{v_{q+1}}_{C^1} + \norm{v_q}_{C^{1}})^{\beta'} \notag\\
&\lesssim \; \sum_{q=0}^{\infty}M^{1-\beta'}\lambda_1^{3\beta \frac{1-\beta'}{2}} \lambda_{q+1}^{- \beta \frac{1-\beta'}{2}}\lambda_{q+1}^{4\beta'} \notag \\
&\lesssim  M^{1-\beta'}\lambda_1^{3\beta \frac{1-\beta'}{2}} \, ,
\label{eq:H:beta'}
\end{align}
for $\beta'< \sfrac{\beta}{(8+\beta)}$,
and hence the sequence
$\{v_q\}_{q\geq 0}$
is uniformly bounded $C^0_t H^{\beta'}_x$, for such $\beta'$. Furthermore, by taking  $a$ sufficiently large (depending on $b$, $\beta$ and $\beta'$) the implicit constant in \eqref{eq:H:beta'} can be made to be universal. From \eqref{e:NSE_reynolds}, \eqref{eq:R:q+1:ind},  the previously established uniform boundedness in $C_t^0 L^2_x$, and the embedding $W^{2,1}_x \subset L^2_x$ we obtain that
\begin{align}
\norm{\partial_t v_q}_{H^{-3}} 
&\les \norm{\Proj \div (v_q \otimes v_q) - \nu \Delta v_q - \Proj \div \RR_q}_{H^{-3}} \notag\\
&\les \norm{v_q \otimes v_q}_{L^1} + \norm{v_q}_{L^2} + \norm{\RR_q}_{L^1}
\notag\\
&\les M^2\lambda_1^{3\beta}\notag
\end{align}
where $\Proj$ is the Leray projector. Thus, the sequence $\{v_q\}_{q\geq0}$ is uniformly bounded in $C^1_t H^{-3}_x$. 
It follows that for any $0< \beta'' < \beta'$ the sum
\begin{align}
\sum_{q\geq 0} (v_{q+1}-v_{q}) =: v\notag
\end{align}
converges in $C^0_t H^{\beta''}_x$, and since $\norm{\mathring R_q}_{L^1} \to 0$ as $q \to \infty$, $v$ is a $C^0_t H^{\beta''}_x$ weak solution of the Navier-Stokes equation. Lastly, in view of \eqref{eq:energy_ind} we have that the kinetic energy of $v(\cdot,t)$ is given by $e(t)$ for all $t\in [0,T]$, concluding the proof of the theorem.

\subsection{Proof of Theorem~\ref{thm:NSE:Euler}}
Fix $\bar \beta >0$ and a weak solution $u \in C^{\bar \beta}_{t,x}$ to the Euler equation on $[-2T,2T]$. The existence of such solutions is guaranteed in view of the results of~\cite{Isett16,BDLSV17} for $\bar \beta < \sfrac 13$, and for $\bar \beta>1$ from the classical local existence results. Let $M_u = \norm{u}_{C^{\bar \beta}}$. Pick an integer $n\geq 1$.

Choose all the parameters as in Proposition \ref{p:main}, except for $a\geq a_0$, which we may take even larger, depending also on $M_u$ and  $\beta'$ which obeys $0< \beta' < \min(\sfrac{\bar \beta}{2} , \sfrac{\beta}{(8+\beta)})$. We make $a$ even larger, depending also on $\beta'$, so that in view of \eqref{eq:H:beta'} we may ensure that 
\begin{align}
\sum_{q=n}^{\infty}M^{1-\beta'}\lambda_1^{3\beta \frac{1-\beta'}{2}} \lambda_{q+1}^{- \beta \frac{1-\beta'}{2}}\lambda_{q+1}^{4\beta'} \leq \frac{1}{2 C n}
\label{eq:lambda:0:n}
\end{align}
where $C$ is the implicit constant in \eqref{eq:H:beta'}.

Let $\{ \phi_{\eps} \}_{\eps>0}$ be a family of standard compact support (of width $2$) Friedrichs mollifiers on $\R^3$ (space), and $\{\varphi_{\eps}\}_{\eps>0}$ be a family of standard compact support (of width $2$) Friedrichs mollifiers on $\R$ (time). We define 
\begin{align}
v_n = (u \ast_x \phi_{\lambda_n^{-1}}) \ast_t \varphi_{\lambda_n^{-1}}
\notag
\end{align}
to be a mollification of $u$ in space and time, at length scale and time scale $\lambda_n^{-1}$, restricted to the temporal range $[0,T]$. Also, on 
$[0,T]$ define the energy function
\begin{align}
e(t) = \int_{\T^3} |v_n(x,t)|^2 \, dx + \frac{\delta_n}{2}\notag
\end{align}
which ensures that \eqref{eq:energy_ind} and \eqref{eq:zero_reynolds} hold for $q=n$.

Since $u$ is a solution of the Euler equations, there exists a mean-free $p_n$ such that 
\begin{align}
\partial_t v_n + \div(v_n \otimes v_n) + \nabla p_n - \lambda_n^{-2} \Delta v_n = \div (\RR_n)\notag
\end{align}
where $\RR_n$ is the traceless symmetric part of the tensor
\begin{align}\notag
(v_n \otimes v_n) - ((u\otimes u)\ast_x \phi_{\lambda_n^{-1}}) \ast_t \varphi_{\lambda_n^{-1}} - \lambda_n^{-2}  \nabla v_n.
\end{align}
Using a version of the commutator estimate introduced in~\cite{ConstantinETiti94}, which may for instance be found in~\cite[Lemma~1]{CDLSJ12}, 
we obtain that 
\begin{align}
\norm{\RR_n}_{L^1} \les \norm{\RR_n}_{C^0} \les  \lambda_n^{-1} M_u + \lambda_n^{-2\bar \beta} M_u^2 \label{eq:VV:1}.
\end{align}
In addition, from a similar argument it follows that 
\begin{align}
\norm{\RR_n}_{C^{1}_{t,x}} 
&\les  M_u + \lambda_n^{1-2\bar \beta } M_u^2 \label{eq:VV:2}\\
\norm{v_n}_{C^{1}_{t,x}} &\les \lambda_n^{1-\bar \beta } M_u.
\label{eq:VV:3}
\end{align}
Setting 
\[
\nu:=\nu_n:=\lambda_n^{-1} \,,
\] 
then with $a$ sufficiently large, depending on $M_u$ and $\bar\beta$, we may ensure the pair $(v_n,\RR_n)$ obey the inductive assumptions \eqref{e:V_ind}--\eqref{e:R_ind_C1} for $q=n$.  Additionally, we may also choose $a$ sufficiently large, depending on $M_u$ and $\bar\beta$, so that 
\begin{align}
\lambda_n^{\bar \beta - \beta'} M_u \leq \frac{1}{2n |\T^3|^{\sfrac 12}}\,.
\label{eq:eps:Euler:2}
\end{align}
At this stage we may start the inductive Proposition~\ref{p:main}, and as in the proof of Theorem~\ref{thm:compact:support}, we obtain a weak solution $u^{(\nu_n)}$ of the Navier-Stokes equations, with the desired regularity, such that 
\begin{align}\notag
\norm{v^{(\nu_n)} - u}_{H^{\beta'}} \leq \norm{v^{(\nu_n)} - v_n}_{H^{\beta'}} + |\T^3|^{\sfrac 12} \norm{u - v_n}_{C^{\beta'}} \leq \frac{1}{n} \,.
\end{align}
in view of \eqref{eq:lambda:0:n} and \eqref{eq:eps:Euler:2}.
 Since $n$ was arbitrary, this concludes the proof of the theorem.

\section{Intermittent Beltrami Waves}

In this section we will describe in detail the construction of the \emph{intermittent Beltrami waves} which will form the building blocks of our convex integration scheme. Very roughly, intermittent Betrami waves are approximate Beltrami waves (approximate eigenfunctions to the curl operator) whose $L^1$ norm is significantly smaller than their $L^2$ norm.

\subsection{Beltrami waves}

We first recall from Proposition 3.1 and Lemma 3.2 in~\cite{DLSZ09} the construction of Beltrami waves (see also the summary given in~\cite{BDLISJ15}). In order to better suit our later goal of defining intermittent Beltrami waves, the statements of these propositions are slightly modified from the form they appear in \cite{BDLISJ15}, by making the substitution $\frac{k}{|k|} \mapsto \xi$. 

\begin{proposition}\label{p:Beltrami}
Given $\xi\in {\mathbb S}^2\cap {\mathbb Q}^3$, let $A_{\xi}\in {\mathbb S}^2 \cap {\mathbb Q}^3$ be such that 
$$
A_{\xi}\cdot {\xi}=0,~|A_{\xi}|=1,~A_{-{\xi}}=A_{\xi}~.
$$
Furthermore, let 
$$
B_{\xi}= \tfrac{1}{\sqrt{2}} \left(A_{\xi}+i\xi\times A_{\xi}\right)~.
$$
Let $\Lambda$ be a given finite subset of $\mathbb S^2\cap\mathbb Q^3$ such that $-\Lambda = \Lambda$, and let $\lambda\in\mathbb Z$ be such that $\lambda\Lambda\subset \mathbb Z^3$. Then for any choice of coefficients $a_{\xi}\in\CC$ with $\overline{a}_{\xi} = a_{-\xi}$ the vector field
\begin{equation}\label{e:Beltrami}
W(x) = \sum_{\xi\in \Lambda}
a_{\xi} B_{\xi}e^{i\lambda  \xi\cdot x}
\end{equation}
is real-valued, divergence-free and satisfies
\begin{equation}\label{e:Bequation}
\div (W\otimes W)=\nabla\frac{|W|^2}{2}.
\end{equation}
Furthermore, since $B_\xi \otimes B_{-\xi} + B_{-\xi} \otimes B_{\xi} = \Id - \xi \otimes \xi$, we have
\begin{equation}\label{e:av_of_Bel}
\fint_{\T^3} W\otimes W\,dx = \frac{1}{2} \sum_{\xi\in\Lambda} |a_{\xi}|^2 \left( \Id - \xi \otimes\xi\right)\, .  
\end{equation}
\end{proposition}
\begin{proposition}\label{p:split}
For every $N\in\N$ we can choose $\eps_\gamma>0$ and $\lambda > 1$ with the following property. Let $B_{\eps_\gamma}(\Id)$ denote the ball of symmetric $3\times 3$ matrices, centered at $\Id$, of radius $\eps_\gamma$. Then, there exist pairwise disjoint subsets 
$$
\Lambda_{\alpha}\subset\mathbb S^2 \cap {\mathbb Q}^3\qquad \alpha\in \{1, \ldots, N\} \, ,
$$
with $\lambda \Lambda_\alpha \in \Z^3$, and smooth positive functions  \[
\gamma_{ {\xi}}^{ {(\alpha)}}\in C^{\infty}\left(B_{\eps} (\Id)\right) \qquad \alpha\in \{1,\dots, N\}, \, \xi\in\Lambda_{\alpha} \,,
\]
with derivatives that are bounded independently of $\lambda$, 
such that:
\begin{itemize}
\item[(a)] $\xi\in \Lambda_{\alpha}$ implies $-\xi\in \Lambda_{\alpha}$ and $\gamma_{\xi}^{ {(\alpha)}} = \gamma_{-\xi}^{{(\alpha)}}$;
\item[(b)] For each $R\in B_{\eps_\gamma} (\Id)$ we have the identity
\begin{equation}\label{e:split}
R = \frac{1}{2} \sum_{\xi\in\Lambda_{\alpha}} \left(\gamma_{\xi}^{ {(\alpha)}}(R)\right)^2 \left(\Id - \xi\otimes \xi\right) .
\end{equation}
\end{itemize}
\end{proposition}

\begin{remark}
\label{rem:lcm}
Throughout the construction, the parameter $N$ is bounded by a universal constant; for instance one can take $N=2$. Moreover, for each $\alpha$ the cardinality of the set $\Lambda_\alpha$ is also bounded by a universal constant; for instance one may take $|\Lambda_\alpha| = 12$. Consequently, the set of direction vectors $\cup_{\alpha=1}^{N} \cup_{\xi \in \Lambda_\alpha} \{ \xi, A_\xi, \xi \times A_\xi\} \subset {\mathbb S}^2 \cap {\mathbb Q}^3$ also has a universally bounded cardinality. Therefore, there exists a universal sufficiently large natural number $N_{\Lambda} \geq 1$ such that we have 
\[
\left\{N_{\Lambda} \xi, N_{\Lambda} A_\xi, N_{\Lambda} \xi \times A_\xi\right \}\subset N_{\Lambda}{\mathbb S}^2 \cap {\mathbb Z}^3
\] 
for all vectors $\xi$ in the construction.

It is also convenient to introduce a sufficiently small geometric constant $c_\Lambda \in (0,1) $ such that
\[
\xi + \xi' \neq 0 \quad \Rightarrow \quad \abs{\xi + \xi'} \geq 2 c_\Lambda
\]
for all $\xi, \xi' \in \Lambda_\alpha$ and all $\alpha \in \{1,\ldots,N\}$. In view of the aforementioned cardinality considerations, the geometric constant $c_\Lambda$ is universal to the construction.

The implicit constants in the $\les$ of the below estimates are  allowed to depend on $N_{\Lambda}$ and $c_{\Lambda}$, but we will not emphasize this dependence, since these are universal constants. 
\end{remark}

\subsection{Intermittent Beltrami waves}\label{s:intermittent Beltrami}

Recall cf.~\cite[Section 3]{GrafakosClassical} that the Dirichlet kernel $D_n$ is defined as
\begin{equation}\label{e:Dirichlet_def}
D_n(x) = \sum_{\xi=-n}^n e^{i x \xi} = \frac{\sin((n+\sfrac12)x)}{\sin(x/2)}
\end{equation}
and has the property that for any $p>1$ it obeys the estimate
\begin{align}\notag
\norm{D_n}_{L^p} \sim n^{1-\sfrac{1}{p}} 
\end{align}
where the implicit constant only depends only on $p$. 
Replacing the sum in \eqref{e:Dirichlet_def}  by a sum of frequencies in a 3D integer cube
\begin{equation}
\Omega_{r}:=\Big\{\xi = (j,k,\ell) \colon j,k,\ell \in \{-r,\dots, r\} \Big\}\notag
\end{equation}
and normalizing to unit size in $L^2$, we obtain a kernel
\[
D_{r}(x) :=\frac{1}{(2r+1)^{\sfrac{3}{2}}}\sum_{\xi\in \Omega_r} e^{i \xi \cdot x} = \frac{1}{(2r+1)^{\sfrac{3}{2}}}\sum_{j,k,\ell \in \{-r,\dots, r\} } e^{i (j x_1 + k x_2 + \ell x_3)} 
\]
such that for $1< p \leq \infty$ we have
\begin{align}
\norm{D_{r}}_{L^2}^2 = (2\pi)^3, 
\qquad \mbox{and} \qquad 
\norm{D_{r}}_{L^p}\lesssim r^{\sfrac{3}{2}-\sfrac{3}{p}}\,,
\label{eq:Dr:norm}
\end{align}
where the implicit constant depends only on $p$. Note that $- \Omega_r = \Omega_r$.

The principal idea in the construction of intermittent Beltrami waves is to modify the Beltrami waves of the previous section by adding oscillations that mimic the structure of the kernels $D_{r}$ in order to construct approximate Beltrami waves with small $L^{p}$ norm for $p$ close to $1$. The large parameter {\em $r$ will parameterize the number of frequencies along edges of the cube } $\Omega_r$. We  introduce a small parameter $\sigma$, such that $\lambda \sigma \in \N$ {\em parameterizes the spacing between frequencies}, or equivalently such that the resulting rescaled kernel is $(\sfrac{\T}{\lambda \sigma})^3$-periodic.  We assume throughout the paper that 
\begin{equation}\label{e:sigma_r_ineq}
\sigma r \leq \sfrac{c_\Lambda}{(10 N_{\Lambda})} \,,
\end{equation}
where $c_\Lambda \in (0,1)$ and $N_{\Lambda} \geq 1$  are  the parameters from Remark~\ref{rem:lcm}. Lastly, we introduce a large parameter $\mu \in (\lambda, \lambda^2)$, which  {\em measures the amount of temporal oscillation} in our building blocks. The parameters $\lambda, r, \sigma$ and $\mu$ are chosen in Section~\ref{sec:perturbation} below.

We recall from Propositions~\ref{p:Beltrami} and~\ref{p:split} that for $\xi \in \Lambda_\alpha$, the vectors $\{ \xi, A_\xi, \xi\times A_\xi\} $ form an orthonormal basis of $\R^3$, and by Remark~\ref{rem:lcm} we have
\[
N_{\Lambda} \xi, ~N_{\Lambda} A_\xi, ~N_{\Lambda} \xi \times A_\xi \in \Z^3\qquad \mbox{for all} \qquad \xi \in \Lambda_\alpha, \alpha \in \{1,\ldots,N\}\, .
\] 
Therefore, for $\xi \in \Lambda_\alpha^+$ we may define a {\em directed and rescaled $(\sfrac{\T}{\lambda \sigma})^3 = (\sfrac{\R}{2\pi \lambda \sigma \Z})^3$-periodic Dirichlet kernel} by
\begin{align}
\eta_{(\xi)}(x,t) = \eta_{\xi,\lambda,\sigma,r,\mu}(x,t) = D_r\left(\lambda \sigma N_{\Lambda} (\xi \cdot x + \mu t), \lambda \sigma N_{\Lambda} A_\xi \cdot x, \lambda \sigma N_{\Lambda} (\xi \times A_\xi)\cdot x\right) \, .
\label{eq:eta:oxi}
\end{align}
For $\xi \in \Lambda_\alpha^-$ we define $\eta_{(\xi)}(x,t) := \eta_{(-\xi)}(x,t)$.
The periodicity of $\eta_{(\xi)}$ follows from the fact that $D_r$ is $\T^3$-periodic, and the definition  of $N_{\Lambda}$. We emphasize that we have the important identity
\begin{align}
\frac{1}{\mu}\partial_t \eta_{(\xi)}(x,t) = \pm (\xi\cdot\nabla)  \eta_{(\xi)}(x,t), \quad \mbox{for all} \quad \xi \in \Lambda_\alpha^{\pm}
\label{eq:kind:of:magic}
\end{align}
as a consequence of the fact that the vectors $A_{\xi}$ and $\xi \times A_{\xi}$ are orthogonal to $\xi$. 

Note that the map $(x_1,x_2,x_3) \mapsto \left(\lambda \sigma N_{\Lambda} (\xi \cdot x + \mu t), \lambda \sigma N_{\Lambda} A_\xi \cdot x, \lambda \sigma N_{\Lambda} (\xi \times A_\xi)\cdot x\right)$ is the composition of a  rotation by a rational orthogonal matrix which maps $(e_1,e_2,e_3)$ to $(\xi, A_\xi, \xi \times A_\xi)$, a rescaling by $\lambda \sigma N_{\Lambda}$, and  a translation by $\lambda \sigma N_{\Lambda} \mu t e_1$. These are all volume preserving transformations on $\T^3$, and thus by our choice of normalization for \eqref{eq:Dr:norm} we have that 
\begin{align}
\fint_{\T^3} \eta_{(\xi)}^2 (x,t) dx  = 1, \qquad \mbox{and} \qquad \norm{\eta_{(\xi)}}_{L^p(\T^3)} \les r^{3/2-3/p}
\label{eq:eta:oxi:norm}
\end{align}
for all $1<p\leq \infty$, pointwise in time.

Letting $W_{(\xi)}$ be the Beltrami plane wave at frequency $\lambda$, namely
\begin{align}
\label{eq:W:oxi:def}
W_{(\xi)}(x) = W_{\xi,\lambda}(x) = B_{\xi} e^{i \lambda \xi \cdot x} \, ,
\end{align}
we have 
\begin{align*}
\curl W_{(\xi)} = \lambda W_{(\xi)} \qquad \mbox{and}\qquad \div W_{(\xi)} = 0\,.
\end{align*}
We take $\lambda$ to be a multiple of $N_\Lambda$, so that $W_{(\xi)}$ is $\T^3$-periodic.
Finally, we define the {\em intermittent Beltrami wave} $\WW_{(\xi)}$ as
\begin{align}
\label{e:Dirichlet_Beltrami_def}
\WW_{(\xi)}(x,t) = \WW_{\xi,\lambda,\sigma,r,\mu}(x,t) = \eta_{\xi,\lambda,\sigma,r,\mu}(x,t) W_{\xi,\lambda}(x)  = \eta_{(\xi)}(x,t) W_{(\xi)}(x).
\end{align}

We first make a few comments concerning the frequency support of $\WW_{(\xi)}$. In view of \eqref{e:sigma_r_ineq} and the definition of $\eta_{(\xi)}$, which yields $\Proj_{\leq 2 \lambda \sigma r N_{\Lambda}} \eta_{(\xi)} = \eta_{(\xi)}$, we have that 
\begin{align}
\Proj_{\leq 2\lambda } \Proj_{\geq \sfrac{\lambda }{2}} \WW_{(\xi)} = \WW_{(\xi)} \, ,
\label{eq:frequency:support}
\end{align}
while for $\xi' \neq -\xi$, by the definition of $c_\Lambda$ in Remark~\ref{rem:lcm} we have
\begin{align}
\Proj_{\leq 4\lambda } \Proj_{\geq c_{\Lambda} \lambda } \left(\WW_{(\xi)} \otimes \WW_{(\xi')}\right) = \WW_{(\xi)}\otimes \WW_{(\xi')}.
\label{eq:frequency:support:2}
\end{align}

Note that the vector $\WW_{(\xi)}$ is not anymore divergence free, nor is it an eigenfunction of $\curl$. These properties only hold to leading order:
\begin{align*}
\norm{\frac{1}{\lambda}  \div \WW_{(\xi)}}_{L^2} &= \frac{1}{\lambda}\norm{B_{\xi} \cdot \nabla \eta_{(\xi)}}_{L^2} \les \frac{\lambda \sigma r}{\lambda} = \sigma r\\
 \norm{\frac{1}{\lambda} \curl \WW_{(\xi)} - \WW_{(\xi)}}_{L^2} &= \frac{1}{\lambda}\norm{\nabla \eta_{(\xi)} \times B_{\xi}}_{L^2} \les \frac{\lambda \sigma r}{\lambda} = \sigma r
\end{align*}
and the parameter $\sigma r$ will be chosen to be small.
Moreover, from Propositions~\ref{p:Beltrami} and \ref{p:split} we have:

\begin{proposition}
\label{p:gamma_def}
Let $\WW_{(\xi)}$ be as defined above, and let $\Lambda_\alpha, \eps_\gamma, \gamma_{(\xi)} = \gamma_{\xi}^{(\alpha)}$ be as in Proposition~\ref{p:split}. If $a_{\xi}\in \CC$ are constants chosen such that $\overline a_{\xi}=a_{-\xi}$, the vector field
\begin{equation*}
\sum_{\alpha}\sum_{\xi \in \Lambda_{\alpha}}a_{\xi}\mathbb W_{(\xi)}(x)
\end{equation*}
is real valued. Moreover, for each $R\in B_{\eps_{\gamma}}(\Id)$ we have the identity
\begin{equation}
\label{e:Beltrami_Cancellation}
\sum_{\xi\in\Lambda_{\alpha}}  \left(\gamma_{(\xi)}(R)\right)^2   \fint_{\T^3} \WW_{(\xi)}\otimes \WW_{(- \xi)} dx  =\sum_{\xi\in\Lambda_{\alpha}}  \left(\gamma_{(\xi)}(R)\right)^2 B_{\xi}\otimes B_{- \xi}=R.
\end{equation}
\end{proposition}
\begin{proof}[Proof of Proposition~\ref{p:gamma_def}]
The first statement follows from the fact that $\eta_{(-\xi)}(x,t) = \eta_{(\xi)}(x,t)$. Identity \eqref{e:Beltrami_Cancellation} follows from \eqref{e:split} upon noting that $2 \Re(B_{\xi}\otimes B_{- \xi}) = \Id - \xi \otimes \xi$, and the normalization \eqref{eq:eta:oxi:norm}.
\end{proof}

For the purpose of estimating the oscillation error in Section~\ref{sec:stress}, it is useful to derive a replacement of identity \eqref{e:Bequation}, in the case of intermittent Beltrami waves. For this purpose, we first recall the vector identity
\begin{align*}
(A\cdot \nabla) B+ (B\cdot \nabla) A= \nabla(A\cdot B)-A\times\curl B-B\times \curl A.
\end{align*}
Hence, for $\xi, \xi' \in \Lambda_\alpha$ we may rewrite
\begin{align}
&\div\left(\WW_{(\xi)}\otimes \WW_{(\xi')}+\WW_{(\xi')}\otimes \mathbb W_{(\xi)}\right) \notag \\
&\quad =\left(W_{(\xi)}\otimes W_{(\xi')}+W_{(\xi')}\otimes W_{(\xi)}\right)\nabla\left(\eta_{(\xi)}\eta_{(\xi')}\right)
+\eta_{(\xi)}\eta_{(\xi')}\left( (W_{(\xi)}\cdot \nabla) W_{(\xi')}+(W_{(\xi')}\cdot \nabla) W_{(\xi)}\right) \notag \\
&\quad =\left( (W_{(\xi')}\cdot \nabla) \left(\eta_{(\xi)}\eta_{(\xi')}\right)\right) W_{(\xi)}+\left( (W_{(\xi)}\cdot \nabla) \left(\eta_{(\xi)}\eta_{(\xi')}\right)\right) W_{(\xi')} \notag \\
&\quad\quad+\eta_{(\xi)}\eta_{(\xi')}\nabla\left(W_{(\xi)}\cdot W_{(\xi')}\right)-\lambda \eta_{(\xi)}\eta_{(\xi')}\left(W_{(\xi)}\times W_{(\xi')}+W_{(\xi')}\times W_{(\xi)}\right) \notag \\
&\quad =\left( (W_{(\xi')}\cdot \nabla) \left(\eta_{(\xi)}\eta_{(\xi')}\right)\right) W_{(\xi)}+\left( (W_{(\xi)}\cdot \nabla)\left(\eta_{(\xi)}\eta_{(\xi')}\right)\right) W_{(\xi')}+\eta_{(\xi)}\eta_{(\xi')}\nabla\left(W_{(\xi)}\cdot W_{(\xi')}\right) .
\label{eq:useful:1}
\end{align}
In the last equality we have used that the cross-product is antisymmetric.

Let us now restrict to the case $\xi+\xi'=0$. Recall that $W_{(\xi)}= \frac{1}{\sqrt 2} (A_{\xi}+i\xi\times A_{\xi})e^{i\lambda \xi\cdot x}$ , $\xi\cdot A_\xi = 0$, and $|A_{\xi}|=1$. Therefore, when $\xi' = -\xi$ the last term on the right side of \eqref{eq:useful:1} is zero, as $W_{(\xi)}\cdot W_{(-\xi)} = 1$. Thus we obtain
\begin{align*}
\div\left(\WW_{(\xi)}\otimes \WW_{(-\xi)}+\WW_{(-\xi)}\otimes \WW_{(\xi)}\right)
&=
\left( (W_{(-\xi)}\cdot \nabla) \eta_{(\xi)}^2\right)W_{(\xi)}+\left( (W_{(\xi)}\cdot \nabla) \eta_{(\xi)}^2\right)W_{(-\xi)}
\notag \\
&=\left( (A_{\xi}\cdot \nabla) \eta_{(\xi)}^2\right) A_{\xi}+\left( ( (\xi\times A_{\xi})\cdot \nabla) \eta_{\xi}^2\right)(\xi\times A_{\xi})
\notag \\
&=\nabla \eta_{(\xi)}^2- \left( (\xi\cdot \nabla) \eta_{(\xi)}^2\right)\xi
\, .
\end{align*}
In the last equality above we have used that $\left\{ \xi, A_\xi, \xi\times A_\xi\right\}$ is an orthonormal basis of $\R^3$.
The above identity and property \eqref{eq:kind:of:magic} of $\eta_{(\xi)}$ shows that 
\begin{align}
\div\left(\WW_{(\xi)}\otimes \WW_{(-\xi)}+\WW_{(-\xi)}\otimes \WW_{(\xi)}\right) = \nabla \eta_{(\xi)}^2 - \frac{\xi}{\mu} \partial_t \eta_{(\xi)}^2.
\label{eq:useful:2}
\end{align}
which is the key identity that motivates the introduction of temporal oscillations in the problem.

Recall, the intermittent Beltrami waves were designed to include additional oscillations that cancel in order to minimize their $L^1$ norm, in a way that is analogous to the cancellations of the Dirichlet kernel. In this direction, an immediate consequence of property \eqref{eq:eta:oxi:norm} of $\eta_{(\xi)}$, of the frequency localization in the spatial variable~\eqref{eq:frequency:support}, and of the frequency of the temporal oscillations, are the following bounds for $\eta_{(\xi)}$ and the the intermittent Beltrami waves $\WW_{(\xi)}$:

\begin{proposition}\label{p:W_bounds}
Let  $ \mathbb W_{(\xi)}$  be defined as above.  The bound
\begin{align}
\norm{\nabla^N \partial_t^K \mathbb W_{(\xi)}}_{L^p}&\lesssim  \lambda^N (\lambda\sigma r\mu)^K r^{\sfrac{3}{2}-\sfrac3p}\label{e:W_Lp_bnd}\\
\norm{\nabla^N \partial_t^K \eta_{(\xi)}}_{L^p}&\lesssim  (\lambda\sigma r)^N (\lambda\sigma r\mu)^K r^{\sfrac{3}{2}-\sfrac3p}\label{e:eta_Lp_bnd}
\end{align}
for any $1<p\leq \infty$, $N\geq 0$ and $K\geq 0$. The implicit constant  depends only on $N, K$  and $p$.
\end{proposition}

\begin{remark}
We note that while in the above proposition we state estimates for all orders of derivatives ($N$ and $K$), only derivatives up to a fixed order, which is independent of $q$, appear in the entire proof of Proposition~\ref{p:main}.  Hence the implicit constants that depend on the number of derivatives taken are independent of $q$. This remark also applies to estimates in later parts of the paper (e.g.~mollification estimates).
\end{remark}

\subsection{$L^p$ decorrelation} 
We now introduce a crucial lemma  from~\cite{BMV17} that will be used throughout the paper. Suppose we wish to estimate
\begin{equation}\notag
\norm{f \; \mathbb W_{(\xi)}}_{L^1}
\end{equation}
for some arbitrary function $f:\mathbb T^3\rightarrow \mathbb R$. The trivial estimate is 
\begin{align}\notag
\norm{f \; \mathbb W_{(\xi)}}_{L^1}\lesssim \norm{f}_{L^2}\norm{\mathbb W_{(\xi)}}_{L^2}\,.
\end{align}
Such an estimate does not however take advantage of the special structure of the $(2\pi\lambda \sigma)^{-1}$ periodic function $\mathbb W_{(\xi)} e^{-i \lambda \xi \cdot x}$. It turns out that if say $f$ has frequency contained in a ball of radius $\mu$ and $\lambda\sigma\gg \mu$ then one obtains the improved estimate 
\begin{align}\notag
\norm{f \; \mathbb W_{(\xi)}}_{L^1}\lesssim \norm{f}_{L^1}\norm{\mathbb W_{(\xi)}}_{L^1}
\end{align}
which gives us the needed gain because $\norm{\mathbb W_{(\xi)}}_{L^1} \ll \norm{\mathbb W_{(\xi)}}_{L^2}$. 
This idea is one of the key insights of~\cite{BMV17} and is summarized in Lemma~\ref{lem:Lp:independence} below. For convenience we include the proof in Appendix~\ref{app:Lp:product}.

\begin{lemma}
\label{lem:Lp:independence}
Fix integers $M,\kappa,\lambda \geq 1$ such that
\begin{align}\notag
\frac{2\pi\sqrt{3} \lambda}{\kappa} \leq \frac{1}{3} \qquad \mbox{and} \qquad \lambda^{4} \frac{(2\pi\sqrt{3} \lambda)^M}{\kappa^M}  \leq 1\,.
\end{align}
Let $p \in \{1,2\}$, and let $f$ be a $\T^3$-periodic function such that there exists a constant  $C_f$ such that
\begin{align}\notag
\|D^j f\|_{L^p} \leq C_f\lambda^j 
\end{align}
for  all $1 \leq j \leq M+4$. 
In addition, let $g$ be a $(\sfrac{\T}{\kappa})^{3}$-periodic function. Then we have that 
\begin{align}\notag
 \|f g \|_{L^p} \lesssim C_f \|g\|_{L^p} 
\end{align}
holds, where the implicit constant is universal.
\end{lemma}

\section{The perturbation}\label{sec:perturbation}
In this section we will construct the perturbation $w_{q+1}$.

\subsection{Mollification of $v_q$}
\label{sec:mollify}
In order to avoid a loss of derivative, we replace $v_q$ by a mollified velocity field $v_\ell$. 
Let $\{ \phi_{\eps} \}_{\eps>0}$ be a family of standard Friedrichs mollifiers (of compact support of radius $2$) on $\R^3$ (space), and $\{\varphi_{\eps}\}_{\eps>0}$ be a family of standard Friedrichs mollifiers (of compact support of width $2$) on $\R$ (time).
We define a mollification of $v_q$ and $\RR_q$ in space and time, at length scale and time scale $\ell$ (which is defined in \eqref{eq:ell:def} below) by
\begin{align}
v_{\ell} = (v_q \ast_x \phi_{\ell}) \ast_t \varphi_{\ell}  \, ,\notag\\
\RR_{\ell} = (\RR_q \ast_x \phi_{\ell}) \ast_t \varphi_{\ell} \,.
\label{eq:R:ell:def}
\end{align}
Then using \eqref{e:NSE_reynolds} we obtain that $(v_\ell,\RR_\ell)$ obey
\begin{subequations}
\label{e:NSE_reynolds_ell}
\begin{align}
\partial_t  v_\ell + \div(v_\ell \otimes v_\ell) + \nabla p_\ell - \Delta v_\ell 
&= \div \Big(  \RR_\ell + \tilde R_{\rm commutator} \Big)\, ,  \\
\div v_\ell &= 0 \, ,
\end{align}
\end{subequations}
where the new pressure $p_\ell$ and the traceless symmetric commutator stress $\tilde R_{\rm commutator}$ are given by
\begin{align}
\tilde p_\ell &= (p_q \ast_x \phi_{\ell}) \ast_t \varphi_{\ell} - \left(|v_\ell|^2  - (|v_q|^2 \ast_x \phi_\ell) \ast_t \varphi_\ell \right)  \, , \notag\\
\tilde R_{\rm commutator} &= (v_\ell \mathring \otimes v_\ell) - ((v_q \mathring \otimes v_q)\ast_x \phi_{\ell}) \ast_t \varphi_{\ell} \, .\label{e:NSE_reynolds_ell:a} 
\end{align}
Here we have used $a \mathring \otimes b$ to denote the traceless part of the tensor $a \otimes b$. 

Note that in view of \eqref{e:V_ind} the commutator stress $\tilde R_{\rm commutator}$ obeys the lossy estimate
\begin{align}
\norm{ \tilde R_{\rm commutator}}_{L^\infty} \les \ell \norm{v_q\otimes v_q}_{C^{1}}\les \ell \norm{v_q}_{C^{1}}\norm{v_q}_{L^{\infty}}\les \ell \lambda_q^8.
 \label{eq:Rc:bound}
\end{align}
The parameter $\ell$ will be chosen (cf.~\eqref{eq:ell:def} below) to satisfy
\begin{align}
(\sigma \lambda_{q+1})^{-\sfrac 12} \ll \ell \ll  \lambda_q^{-19} \delta_{q+1} 
\label{eq:ell:cond:1}  \,.
\end{align}
In particular, $\RR_\ell$ inherits the $L^1$ bound of $\RR_q$ from \eqref{eq:R:q+1:ind}, and in view \eqref{e:R_ind_C1} and the upper bound on $\ell$ in \eqref{eq:ell:cond:1}, we have that 
\begin{align}
\norm{\RR_\ell}_{C^N_{t,x}} \les \lambda_q^{10}\ell^{-N+1}\les \ell^{-N}\,.
\label{eq:R_N}
\end{align}
Moreover, from \eqref{e:V_ind} and the upper bound on $\ell$ from \eqref{eq:ell:cond:1} we obtain the bounds
\begin{align}
 \norm{v_q - v_\ell}_{L^\infty}& \les  \ell\norm{v_q}_{C^1}\les \ell \lambda_q^4 
 \label{eq:V_ell_est} \, ,\\
  \norm{v_\ell}_{C^N_{x,t}}& \les \ell^{1-N}\norm{v_q}_{C^1}\les \ell^{1-N}\lambda_q^4\les\ell^{-N}\label{eq:V_ell_est_N}\, .
\end{align}

\subsection{Stress cutoffs}
Because the Reynolds stress $\RR_\ell$ is not spatially homogenous, we  introduce stress cutoff functions.
We let $0 \leq \tilde \chi_0, \tilde \chi \leq 1$  be bump functions adapted to the intervals $[0,4]$ and $[1/4,4]$ respectively, such that together they form a partition of unity:
\begin{align}
\tilde \chi_{0}^2(y) + \sum_{i\geq 1}  \tilde \chi^2_{i}(y) \equiv 1, \quad \mbox{where} \quad \tilde \chi_i (y) = \tilde \chi(4^{-i} y),
\label{eq:tilde:partition}
\end{align}
for any $y>0$.
We then define 
\begin{align}
\chi_{(i)}(x,t)=\chi_{i,q+1}(x,t) = \tilde \chi_{i}\left( \left\langle \frac{\RR_{\ell}(x,t)}{100 \lambda_q^{-\eps_R}\delta_{q+1}}\right\rangle\right) 
\label{eq:chi:i:def}
\end{align}
for all $i\geq 0$. 
Here and throughout the paper we use the notation $\langle A \rangle = (1+ |A|^2)^{1/2}$ where $|A|$ denotes the Euclidean norm of the matrix $A$.
By definition the cutoffs $\chi_{(i)}$ form a partition of unity
\begin{align}
\sum_{i\geq 0} \chi_{(i)}^2 \equiv 1
\label{eq:chi:partition} 
\end{align}
and we will show in Lemma~\ref{lem:max:i} below that there exists an index $i_{\rm max} = i_{\rm max}(q)$, such that $\chi_{(i)} \equiv 0$ for all $i > i_{\rm max}$, and moreover that $4^{i_{\rm max}} \les \ell^{-1}$.

\subsection{The definition of the velocity increment}

Define the coefficient function $a_{\xi,i,q+1}$ by
\begin{equation}
a_{(\xi)}:= a_{\xi,i,q+1}:=\rho_i^{\sfrac 12} \chi_{i,q+1} \gamma_{(\xi)}\left(\Id - \frac{\RR_{\ell}}{ \rho_i(t)}\right).
\label{eq:a:oxi:def}
\end{equation}
where for $i\geq 1$, the parameters $\rho_i$ are defined by
\begin{align}\label{e:rho_i_def}
\rho_i:= \lambda_q^{-\eps_R} \delta_{q+1} 4^{i+c_0}
\end{align}
where $c_0 \in \N$ is a sufficiently large constant, which depends on the $\eps_\gamma$ in Proposition~\ref{p:gamma_def}. The addition of the factor $4^{c_0} $ ensures that the argument of $\gamma_{(\xi)}$ is in the range of definition.  The definition $\rho_0$ is slightly more complicated and as such its definition will be delayed to Section \ref{ss:rho_0} below, see~\eqref{eq:rho:0:def:1} and \eqref{eq:rho:0:def}. Modulo the definition of $\rho_0$, we note that as a consequence of \eqref{eq:frequency:support:2}, \eqref{e:Beltrami_Cancellation}, \eqref{eq:chi:partition}, and \eqref{eq:a:oxi:def} we have 
\begin{align}
\sum_{i\geq 0}\sum_{\xi,\xi' \in \Lambda_{(i)}}a_{(\xi)}^2\fint_{\mathbb T^3}\left(\mathbb W_{(\xi)}\otimes \mathbb W_{(\xi')} \right)dx=
\sum_{i\geq 0} \rho_i \chi_{(i)}^2\Id-\RR_{\ell}\,,
\label{e:WW_id}
\end{align}
which justifies the definition of the amplitude functions $a_{(\xi)}$.

By a slight abuse of notation, let us now fix $\lambda, \sigma, r$, and $\mu$ for the short hand notation $\mathbb  W_{(\xi)}$, $W_{(\xi)}$ and $\eta_{(\xi)}$ introduced in Section~\ref{s:intermittent Beltrami} (cf.~\eqref{eq:eta:oxi}, \eqref{eq:W:oxi:def}, \eqref{e:Dirichlet_Beltrami_def}):
\begin{equation}\notag
\mathbb  W_{(\xi)}:=\mathbb W_{\xi,\lambda_{q+1},\sigma ,r ,\mu},\quad W_{(\xi)}:=W_{\xi,\lambda_{q+1}} \quad \mbox{and}\quad \eta_{(\xi)}:=\eta_{\xi,\lambda_{q+1},\sigma ,r ,\mu}\, ,
\end{equation}
where the integer $r$, the parameter $\sigma$, and the parameter $\mu$ are defined by
\begin{align}
r = \lambda_{q+1}^{\sfrac 34}, \qquad \sigma = \lambda_{q+1}^{-\sfrac {15}{16}}\qquad \mbox{and} \qquad \mu=\lambda_{q+1}^{\sfrac 54}\,.
\label{eq:r:sigma:def}
\end{align}
The fact that $\lambda_{q+1} \sigma \in \N$ is ensured by our choices $a\in \N$ and $b \in 16 \N$. In order to ensure $\lambda_{q+1}$ is a multiple of $N_{\Lambda}$, we need to choose $a$ which is a multiple of $N_{\Lambda}$.
Moreover, at this stage we fix
\begin{align}
\ell = \lambda_q^{-20} \, ,
\label{eq:ell:def}
\end{align}
which in view of the choice of $\sigma$ in \eqref{eq:r:sigma:def}, ensures that \eqref{eq:ell:cond:1} holds, upon taking $\lambda_0$ sufficiently large.
In view of \eqref{eq:ell:def}, throughout the rest of the paper we may use either $\ell^{\eps} \leq \lambda_0^{-20 \eps}$ or $\lambda_{q}^{-\eps} \leq \lambda_0^{-\eps}$, with $\eps>0$ arbitrarily small, to absorb any of the constants (which are $q$-independent) appearing due to $\les$ signs in the below inequalities. This is possible by choosing $\lambda_0 =  a$, sufficiently large.

The {\em principal part of $w_{q+1}$} is defined as
\begin{equation}
w^{(p)}_{q+1}:=\sum_{i}\sum_{\xi\in \Lambda_{(i)}}a_{(\xi)} \; \mathbb  W_{(\xi)}\,,
\label{eq:w:q+1:p:def}
\end{equation}
where the sum is over $0 \leq i \leq i_{\rm max}(q)$. The sets $\Lambda_{(i)}$ are defined as follows. In Lemma~\ref{p:split} it suffices to take $N=2$, so that $\alpha \in \{ \alpha_0, \alpha_1 \}$, and we define $\Lambda_{(i)} = \Lambda_{\alpha_{i {\rm \, mod\, } 2}}$. This choice is allowable since $\chi_i \chi_{j} \equiv 0$ for $\abs{i-j}\geq 2$. In order to fix the fact that $w_{q+1}^{(p)}$ is not divergence free, we define an {\em incompressibility corrector} by
\begin{align}
w^{(c)}_{q+1}:=&\frac{1}{\lambda_{q+1}}\sum_{i}\sum_{\xi\in \Lambda_{(i)}}\nabla \left(a_{(\xi)}\eta_{(\xi)} \right)\times W_{(\xi)}\,.
\label{eq:w:q+1:c:def}
\end{align}
Using that  {$\div W_{(\xi)} = 0$}, we then have
\begin{align}\label{e:curl_form}
w_{q+1}^{(p)}+w_{q+1}^{(c)}=\frac{ {1} }{\lambda_{q+1}}\sum_{i}\sum_{\xi\in \Lambda_{(i)}}\curl\left(a_{(\xi)}\eta_{(\xi)} W_{(\xi)}\right)=\frac{ {1} }{\lambda_{q+1}}\curl(w_{q+1}^{(p)})\,,
\end{align}
and thus
\begin{equation}\notag
\div\left(w_{q+1}^{(p)}+w_{q+1}^{(c)}\right)=0\,.
\end{equation}

In addition to the incompressibility corrector $w_{q+1}^{(c)}$, we introduce a  {\em temporal corrector} $w_{q+1}^{(t)}$, which is defined by
\begin{equation}\label{e:temporal_corrector}
w^{(t)}_{q+1}:=\frac{1}{\mu}\sum_{i}\sum_{\xi\in \Lambda_{(i)}^+}\mathbb P_{H}\mathbb P_{\neq 0}\left(a_{(\xi)}^2\eta_{(\xi)}^2 \xi\right)\,.
\end{equation}
Here we have denoted by $\Proj_{\neq 0}$ the operator which projects a function onto its nonzero frequencies $\Proj_{\neq 0} f = f - \fint_{\T^3} f$, and have used $\mathbb P_{H}$ for the usual Helmholtz  (or Leray) projector onto divergence-free vector fields, $\Proj_H f  = f  - \nabla (\Delta^{-1} \div f)$. The purpose of the corrector $w_{q+1}^{(t)}$ becomes apparent upon recalling \eqref{eq:useful:2}. Indeed, if we multiply identity \eqref{eq:useful:2} by $a_{(\xi)}^2$,  remove the mean and a suitable pressure gradient, the leading order term left is $-(\sfrac{1}{\mu}) \Proj_{H} \Proj_{\neq 0} ( \xi a_{(\xi)}^2 \partial_t \eta_{(\xi)}^2)$, see~\eqref{eq:Ronaldo:7} below. This term is not of high frequency (proportional to $\lambda_{q+1}$). Moreover, upon writing this term as the divergence of a symmetric stress, the size of this stress term in $L^1$ is $\delta_{q+1}$, instead of $\delta_{q+2}$; thus this term does not obey a favorable estimate and has to be cancelled altogether. The corrector $w_{q+1}^{(t)}$ is designed such that its {\em time derivative} achieves precisely this goal, of cancelling the $-(\sfrac{1}{\mu}) \Proj_{H} \Proj_{\neq 0} ( \xi a_{(\xi)}^2 \partial_t \eta_{(\xi)}^2)$ term.

Finally, we define the velocity increment $w_{q+1}$ by
\begin{equation}\label{eq:w:q+1:def}
w_{q+1}:=w_{q+1}^{(p)}+w_{q+1}^{(c)}+w_{q+1}^{(t)},
\end{equation}
which is by construction mean zero and divergence-free. 
The new velocity field $v_{q+1}$ is then defined as
\begin{equation}\label{eq:v:q+1:def}
v_{q+1} = v_\ell + w_{q+1} \, .
\end{equation}

\subsection{The definition of $\rho_0$}\label{ss:rho_0}
It follows from \eqref{e:WW_id} that with the $\rho_i$ defined above we have
\begin{align}
\sum_{i\geq 1}\int_{\mathbb T^3}\abs{\sum_{\xi\in \Lambda_{(i)}} a_{(\xi)}\mathbb W_{(\xi)}}^2\,dx
&=\sum_{i\geq 1} \sum_{\xi,   \xi' \in \Lambda_{(i)}} \int_{\mathbb T^3} a_{(\xi)}a_{(\xi')} \tr (\mathbb W_{(\xi)} \otimes \mathbb W_{(-\xi')} )\,dx
\notag\\
&=\sum_{i\geq 1} \sum_{\xi \in \Lambda_{(i)}} \int_{\mathbb T^3}  a^2_{(\xi)}   \tr \fint_{\mathbb T^3}(\mathbb W_{(\xi)} \otimes \mathbb W_{(-\xi)} )\,dx + \mbox{error}
\notag\\
&= 3 \sum_{i\geq 1}\rho_i\int_{\mathbb T^3}\chi_{(i)}^2\, dx  + \mbox{error}\, ,
\label{e:WW_id_almost}
\end{align}
where the error term can be made arbitrarily small since the spatial frequency of the $a_{(\xi)}$'s is $\ell^{-1}$, while the minimal separation of frequencies of $\mathbb W_{(\xi)} \otimes \mathbb W_{-(\xi')} $ is $\lambda_{q+1}\sigma \gg \ell^{-1}$. The term labeled as {\em error} on the right side of \eqref{e:WW_id_almost} above will be estimated precisely in Section~\ref{s:energy} below.
We will show in the next section, Lemma~\ref{l:chi_stuff} that
\begin{align}\label{e:small_energy_contrib}
\sum_{i\geq 1}\rho_i\int_{\mathbb T^3}\chi_{(i)}^2\, dx \lesssim \delta_{q+1}\lambda_{q}^{-\eps_R}\,.
\end{align}
In order to ensure \eqref{eq:energy_ind} is satisfied for $q+1$,  we  design $\rho_0$ such that
\begin{align}\notag
\int_{\mathbb T^3}\abs{\sum_{\xi\in \Lambda_{(0)}} a_{(\xi)}\mathbb W_{(\xi)}}^2\,dx
\approx \widetilde e(t):=e(t)-\int_{\mathbb T^3}\abs{v_q}^2\,dx - 3\sum_{i\geq 1}\rho_i\int_{\mathbb T^3}\chi_{i}^2\, dx\,.
\end{align}
We thus define the auxiliary function
\begin{align} \label{eq:rho:0:def:1}
\rho(t):=\frac{1}{3 \abs{\mathbb T^3}}\left(\int_{\mathbb T^3} \chi_0^2\,dx\right)^{-1}\max\left(\widetilde e(t)-\frac{\delta_{q+2}}{2},\,0\right)\,.
\end{align}
The term $-\sfrac{\delta_{q+2}}{2}$ is added to ensure that we leave room for future corrections and the $\max$ is in place to ensure that we do not correct the energy when the energy of $v_q$ is already sufficiently close to the prescribed energy profile. This later property will allow us to take energy profiles with compact support. Finally, in order to ensure $\rho_0^{\sfrac 12}$ is sufficiently smooth, we define $\rho_0$ as the square of the mollification of $\rho^{\sfrac 12}$ at time scale $\ell$
\begin{align}
\label{eq:rho:0:def}
\rho_0= \left( (\rho^{\sfrac 12})  \ast_t \varphi_{\ell} \right)^2\,.
\end{align}
We note that \eqref{eq:energy_ind} and \eqref{e:chi_lower_bnd} below imply that
\begin{equation}\label{e:rho0_bnd}
\norm{\rho_0}_{C^0_t} \leq 2\delta_{q+1} \quad\mbox{and}\quad \norm{\rho_0^{\sfrac 12}}_{C^{N}_{t}}\les \delta_{q+1}^{\sfrac 12}\ell^{-N} 
\end{equation}
for $N \geq 1$.
By a slight abuse of notation we will denote
\begin{align}\notag
\frac{\RR_{\ell}}{ \rho_0(t)}=\begin{cases}\frac{\RR_{\ell}}{ \rho_0(t)}&\mbox{if }  \chi_0\neq 0\mbox{ and }\RR_{\ell}\neq 0 \, ,\\
0&\mbox{otherwise}\,.
\end{cases}
\end{align}
Observe that if $\chi_0\neq 0$ and $\RR_{\ell}\neq 0$, then \eqref{eq:zero_reynolds} and \eqref{e:small_energy_contrib} ensure that $\rho_0>0$. In order to ensure that $\Id - \frac{\RR_{\ell}}{ \rho_0(t)}$ is in the domain of the functions $\gamma_{(\xi)}$ from Proposition~\ref{p:gamma_def}, we will need to ensure that 
\begin{align}\label{e:rho_gamma_cond}
\norm{\frac{\RR_{\ell}}{ \rho_0(t)}}_{L^{\infty}(\supp \chi_{(0)})}\leq \eps_{\gamma} \, .
\end{align}
We give the proof of~\eqref{e:rho_gamma_cond} next.  
Owing to the estimate
\begin{align}\notag
\abs{e(t)-\int_{\mathbb T^3}\abs{v_q(x,t)}^2 \, dx-e(t')-\int_{\mathbb T^3}\abs{v_q(x,t')}^2} \, dx \les \ell^{\sfrac12}
\end{align}
for $t'\in(t-\ell,t+\ell)$ which follows from Lemma \ref{l:energy_useful} in Section \ref{s:energy}, and the inequality $\ell^{\sfrac 12} \ll \delta_{q+1}$, we may apply  \eqref{eq:zero_reynolds} to conclude that it is sufficient to check the above condition when
\begin{align}\notag
e(t)-\int_{\mathbb T^3}\abs{v_q(x,t)}^2\,dx \geq \frac{\delta_{q+1}}{200}\,.
\end{align}
Then by \eqref{e:small_energy_contrib}, the above lower bound implies
\begin{align}\notag
\tilde e(t) \geq \frac{\delta_{q+1}}{400}\,.
\end{align}
and thus
\begin{align}\notag
\rho(t)\geq \frac{1}{\abs{\mathbb T^3}}\left(\frac{\delta_{q+1}}{400}-\frac{\delta_{q+2}}{2}\right)\geq\frac{\delta_{q+1}}{500}
\end{align}
where we used \eqref{e:chi_lower_bnd} from Lemma \ref{l:chi_stuff} below to bound the integral. Finally, using the estimate \eqref{e:rho_diff} from Section \ref{s:energy} we obtain
\begin{equation*}
\rho_0(t) \geq \frac{\delta_{q+1}}{600} \, .
\end{equation*} Since on the support of $\chi_0$ we have $\abs{\RR_{\ell}}\leq 1000 \lambda_{q}^{-\eps_R}\delta_{q+1}$ we obtain \eqref{e:rho_gamma_cond}.

\subsection{Estimates of the perturbation}
We first collect a number of estimates concerning the cutoffs $\chi_{(i)}$ defined in \eqref{eq:chi:i:def}.
\begin{lemma}
\label{lem:max:i}
For $q\geq 0$, there exists $i_{\rm max}(q) \geq 0$, determined by \eqref{eq:i:max:def} below, such that
\begin{align}\notag
\chi_{(i)} \equiv 0 \quad \mbox{for all} \quad i > i_{\rm max}.
\end{align}
Moreover, we have that  for all $0 \leq i\leq i_{\rm max}$
\begin{align}
\rho_i\les 4^{i_{\rm max}} \les \ell^{-1}
\label{eq:i:max:bound}
\end{align}
where the implicit constants can be made independent of other parameters.
Moreover, we have 
\begin{equation}
\sum_{i= 0}^{i_{\rm max}} \rho_i^{\sfrac12}2^{-i}\leq 3 \delta_{q+1}^{\sfrac 12} \, .\label{e:rho_sum}
\end{equation}
\end{lemma}
\begin{proof}[Proof of Lemma~\ref{lem:max:i}]
Let $i \geq 1$. By the definition of $\tilde \chi_{i}$ we have that $\chi_{(i)} = 0$ for all $(x,t)$ such that 
\[
\langle 100^{-1}\lambda_q^{\eps_R} \delta_{q+1}^{-1} \RR_{\ell}(x,t)\rangle < 4^{i-1} \, .
\]
Using the inductive assumption \eqref{e:R_ind_C1}, we have that 
\[
\norm{\RR_\ell}_{L^\infty} \les \norm{\RR_\ell}_{C^1} \les \norm{\RR_q}_{C^1} \leq C_{\rm max} \lambda_{q}^{10} \leq \lambda_{q}^{10 +\eps_R}
\]
since the implicit constant $C_{\rm max}$ is independent of $q$ and of $\eps_R$ (it only depends on norms of the mollifier $\phi$ used to define $\RR_{\ell}$), and thus we have $C_{\rm max} \leq \lambda_q^{\eps_R}$. Therefore, if  $i\geq 1$ is large enough such that 
\[
\langle 100^{-1}  \delta_{q+1}^{-1}  \lambda_{q}^{10 +2 \eps_R} \rangle \leq 4^{i-2} \,,
\]
then $\chi_{(i)} \equiv 0$. Therefore, as $\eps_R \leq \sfrac 14$
and since we have 
\[
\langle 100^{-1}  \delta_{q+1}^{-1}  \lambda_{q}^{10 +2 \eps_R} \rangle \leq \langle   \delta_{q+1}^{-1}  \lambda_{q}^{10 +1/2} \rangle  \leq \delta_{q+1}^{-1}  \lambda_{q}^{11}
\]
for all $q\geq 0$ (since $\beta b$ is small), we may define $i_{\rm max}$ by
\begin{align}
i_{\rm max}(q) = \min \left\{ i \geq 0 \colon 4^{i-2} \geq \lambda_{q}^{11} \delta_{q+1}^{-1} \right\}.
\label{eq:i:max:def}
\end{align}
Observe that, the first inequality of  \eqref{eq:i:max:bound} follows trivially from the definition of $\rho_i$ for $i\geq 1$ and \eqref{e:rho0_bnd} for $i=0$. The second inequality follows from the fact that 
$\lambda_{q}^{11} \delta_{q+1}^{-1} \leq \ell^{-1}$, which is a consequence of $ b \beta$ being small. Finally, from the definition \eqref{e:rho_i_def} and the bound \eqref{e:rho0_bnd} we have
\[
\sum_{i=0}^{i_{\rm max}} \rho_i^{\sfrac12}2^{-i}\leq 2\delta_{q+1}^{\sfrac 12}+ 2^{c_0} \sum_{i= 1}^{i_{\rm max}} \lambda_q^{-\sfrac{\eps_R}{2}}\delta_{q+1}^{\sfrac 12}  \leq  \delta_{q+1}^{\sfrac 12} \left(2 + 2^{c_0} \lambda_q^{-\sfrac{\eps_R}{2}} (3 + \log_4(\lambda_{q}^{11} \delta_{q+1}^{-1}) ) \right).
\]
Since $\lambda_q^{11} \delta_{q+1}^{-1} \leq \lambda_q^{20}$, which is a consequence of $\beta b $ being small, we can bound the second term above as
\[
2^{c_0} \lambda_q^{-\sfrac{\eps_R}{2}} (3 + \log_4(\lambda_{q}^{11} \delta_{q+1}^{-1}) )
\leq 2^{c_0} \lambda_q^{-\sfrac{\eps_R}{2}} (3 + 20 \log_4 (\lambda_q) ) \,
\leq 1
\]
by taking $a$ (and hence $\lambda_q$) to be sufficiently large, depending on $\eps_R$ and $c_0$.
This finishes the proof of \eqref{e:rho_sum}.
\end{proof}

The size and derivative estimate for the $\chi_{(i)}$ are summarized in the following lemma
\begin{lemma}
\label{lem:psi:bounds}
Let $0 \leq i \leq i_{\rm max}$. Then we have
\begin{align} 
\norm{ \chi_{(i)} }_{L^2} &\les  2^{-i} \label{eq:psi:i:L2}  \\
\norm{ \chi_{(i)}}_{C^{N}_{x,t}} &\les \lambda_q^{10}\ell^{1-N} \les \ell^{-N} \label{eq:psi:i:CN}
\end{align}
for all $N\geq 1$. 
\end{lemma}
\begin{proof}[Proof of Lemma~\ref{lem:psi:bounds}]
We prove that 
\begin{align}\notag
\norm{ \chi_{(i)} }_{L^1} \les  4^{-i} \, ,
\end{align}
so that the bound \eqref{eq:psi:i:L2} follows since $\chi_{(i)}\leq 1$, upon interpolating the $L^2$ norm between the $L^1$ and the $L^\infty$ norms.

When $i=0,1$, we have that $\norm{\chi_{(i)}}_{L^1} \leq |\T^3| \norm{\chi_{(0)}}_{L^\infty}  \les 1 \les 4^{-i}$. For $i \geq 2$, we use the Chebyshev's inequality and the inductive assumption~\eqref{eq:R:q+1:ind} to conclude  
\begin{align}
\norm{\chi_{(i)}}_{L^1} 
&\leq \sup_t \left| \left\{ x \colon 4^{i-1} \leq \langle \lambda_q^{\eps_R} \delta_{q+1}^{-1} \RR_{\ell}(x,t)/100 \rangle \leq 4^{i+1} \right\} \right| \notag\\
&\leq \sup_t \left| \left\{ x \colon 4^{i-1} \leq  \lambda_q^{\eps_R}  \delta_{q+1}^{-1} |\RR_{\ell}(x,t) | /100 + 1  \right\} \right|
\notag\\
&\leq \sup_t \left| \left\{x  \colon  100 \lambda_q^{-\eps_R}  \delta_{q+1} 4^{i-2}  \leq  |\RR_{\ell}(x,t) |  \right\} \right| \notag\\
&\les\lambda_q^{\eps_R}  \delta_{q+1}^{-1} 4^{-i} \norm{\RR_{\ell}}_{L^\infty_t L^1_x} 
\notag\\
&\les \lambda_q^{\eps_R}  \delta_{q+1}^{-1} 4^{-i} \norm{\RR_q}_{L^\infty_t L^1_x} \les 4^{-i}\notag
\end{align}
proving the desired $L^1$ bound. In order to prove the estimate \eqref{eq:psi:i:CN} we appeal to \cite[Proposition C.1]{BDLISJ15} which yields
\begin{align}
\norm{\chi_{(i)}}_{C^{N}_{t,x}} 
&\les \norm{\langle \lambda_q^{\eps_R} \delta_{q+1}^{-1}\RR_{\ell} /100 \rangle}_{C^N_{t,x}} + \norm{\langle \lambda_q^{\eps_R} \delta_{q+1}^{-1}\RR_{\ell} /100\rangle}_{C^1_{t,x}}^N 
\notag\\
&\les \ell^{-N+1} \norm{\RR_{\ell}}_{C^1_{t,x}} + \norm{\RR_{\ell}}_{C^1_{t,x}}^N\notag\\
&\les \ell^{-N+1} \norm{\mathring R_{q}}_{C^1_{t,x}} + \norm{\mathring R_{q}}_{C^1_{t,x}}^N\notag\\
&\les \lambda_q^{10} \ell^{-N+1} +\lambda_q^{10N}
 \les \lambda_q^{10}\ell^{1-N}\notag
\end{align}
where we have used that $\delta_{q+1} \les 1$ and \eqref{e:R_ind_C1}.
\end{proof}

\begin{lemma}\label{l:chi_stuff}
We have that the following lower and upper bounds hold:
\begin{align}
\int_{\T^3} \chi_{(0)}^2 \, dx &\geq \frac{|\mathbb T^3|}{2} \label{e:chi_lower_bnd}\\
\sum_{i\geq 1}\rho_i\int_{\T^3}\chi_{(i)}^2(x,t)\, dx&\les \lambda_{q}^{-\eps_R}\delta_{q+1}\label{e:rho_i_contrib}
\end{align}
\end{lemma}
\begin{proof}[Proof of Lemma~\ref{l:chi_stuff}]
By Chebyshev's inequality we have 
\begin{align}
\abs{\left\{x | \abs{\RR_{\ell}}\geq 2\lambda_{q}^{-\eps_R}\delta_{q+1}
\abs{\mathbb T^3}^{-1}\right\}}\leq \frac{\abs{\mathbb T^3}\norm{\RR_{\ell}}_{L^1}}{2\lambda_{q}^{-\eps_R}\delta_{q+1}}
\leq \frac{\abs{\mathbb T^3}\norm{\mathring R_{q}}_{L^1}}{2\lambda_{q}^{-\eps_R}\delta_{q+1}}
\leq \frac{\abs{\mathbb T^3}}{2}\notag
\end{align}
where we have used \eqref{eq:R:q+1:ind}. Then from the definition of $\chi_{(0)}$ we obtain \eqref{e:chi_lower_bnd}.

Observe that by definition, 
\begin{align}
\sum_{i\geq 1} \rho_i \int_{\T^3} \chi_{(i)}^2 dx 
&\les \sum_{i \geq 1} (4^{i} \lambda_q^{-\eps_R} \delta_{q+1}) \tilde \chi^2 \left(\frac{1}{4^i} \left\langle \frac{\RR_{\ell}}{100 \lambda_q^{-\eps_R} \delta_{q+1}}\right\rangle \right) dx \notag\\
&\les \norm{\RR_\ell}_{L^1} \les \norm{\RR_q}_{L^1} \les \lambda_{q}^{-\eps_R} \delta_{q+1}\notag
\end{align}
from which we conclude \eqref{e:rho_i_contrib}.
\end{proof}

\begin{lemma}
\label{lem:many:bounds}
The bounds
\begin{align}
\norm{a_{(\xi)}}_{L^2} &\les \rho_i^{\sfrac 12} 2^{-i}\les \delta_{q+1}^{\sfrac 12}  \, ,\label{e:a_est_L2} \\
\norm{a_{(\xi)}}_{L^\infty} &\les \rho_i^{\sfrac 12}  \les \delta_{q+1}^{\sfrac 12} 2^i \, ,\label{e:a_est} \\
\norm{a_{(\xi)}}_{C^{N}_{x,t}} &\les  \ell^{-N} \label{e:a_est_CN}
\end{align}
hold for all $0 \leq i \leq i_{\rm max}$ and $N \geq 1$.
\end{lemma}
\begin{proof}[Proof of Lemma~\ref{lem:many:bounds}]
The bound \eqref{e:a_est} follows directly from the definitions \eqref{eq:a:oxi:def}, \eqref{e:rho_i_def} together with the boundedness of the functions $\gamma_{(\xi)}$ and $\rho_0$ given in \eqref{e:rho0_bnd}. Using additionally \eqref{eq:psi:i:L2}, the estimate \eqref{e:a_est_L2} follows similarly.
For \eqref{e:a_est_CN}, we apply   derivatives to \eqref{eq:a:oxi:def}, use~\cite[Proposition C.1]{BDLISJ15}, estimate~\eqref{eq:psi:i:CN}, Lemma~\ref{lem:max:i}, the bound \eqref{eq:R_N} for $\RR_{\ell}$,  the definition \eqref{eq:ell:def} of $\ell$, and the bound \eqref{eq:i:max:bound}, to obtain for the case $i\geq 1$  the estimate
\begin{align}
\norm{a_{(\xi)}}_{C^N_{x,t}}
&\les \rho_i^{\sfrac12}  \left( \norm{\chi_{(i)}}_{L^\infty} \norm{  \gamma_{(\xi)} \left(\Id - \rho_i^{-1} \RR_{\ell} \right) }_{C^N_{x,t}} + \norm{\chi_{(i)}}_{C^N_{x,t}} \norm{  \gamma_{(\xi)} \left(\Id -  \rho_i^{-1}\RR_{\ell}\right) }_{L^\infty }\right)
\notag\\
&\les \rho_i^{\sfrac12}  \left( \rho_i^{-1} \norm{\RR_{\ell}}_{C^N_{x,t}} +  \rho_i^{-1}\norm{\RR_{\ell}}_{C^1_{t,x}}^{N}  + \lambda_q^{10} \ell^{1-N} \right)
\notag\\
&\les\rho_i^{\sfrac12} \left(  \rho_i^{-1}\ell^{-N+1} \norm{\RR_{\ell}}_{C^1_{x,t}} +  \rho_i^{-1}\norm{\RR_{\ell}}_{C^1_{t,x}}^{N}  + \rho_i^{-\sfrac 12} \ell^{-N} \right)
\notag\\
&\les\rho_i^{\sfrac12} \left(  \rho_i^{-1}\ell^{-N+1}\lambda_q^{10} +  \rho_i^{-1}\lambda_q^{10N}  + \rho_i^{-\sfrac 12} \ell^{-N} \right)
\notag\\
&\les  \ell^{-N} \, ,\notag
\end{align}
For  $i=0$, the time derivative may land on $\rho_0^{\sfrac 12}$. We use  \eqref{e:rho0_bnd} to estimate this contribution similarly, by paying an $\ell^{-1}$ per time derivative.
\end{proof}

\begin{proposition}
\label{prop:perturbation}
The principal part and of the velocity perturbation, the incompressibility, and the temporal correctors obey the bounds 
\begin{align}
\norm{w_{q+1}^{(p)}}_{L^2}  
&\leq \frac{M}{2} \delta_{q+1}^{\sfrac 12} \label{e:wp_est}\\
\norm{w_{q+1}^{(c)}}_{L^2} +\norm{w_{q+1}^{(t)}}_{L^2} 
&\les  r^{\sfrac 32}\ell^{-1}\mu^{-1} \delta_{q+1}^{\sfrac 12} \label{e:wc_est}\\
\norm{w_{q+1}^{(p)}}_{W^{1,p}}+\norm{w_{q+1}^{(c)}}_{W^{1,p}}+\norm{w_{q+1}^{(t)}}_{W^{1,p}}
&\les \ell^{-2} \lambda_{q+1} r^{\sfrac 32 - \sfrac 3p} \label{e:wp_est_deriv}\\
\norm{\partial_t w_{q+1}^{(p)}}_{L^{p}} +\norm{\partial_t w_{q+1}^{(c)}}_{L^{p}} 
&\les \ell^{-2}\lambda_{q+1}\sigma \mu r^{\sfrac52-\sfrac 3p}  \label{e:wpc_temporal_est_deriv}\\
 \norm{w_{q+1}^{(p)}}_{C^{N}_{x,t}}  + \norm{w_{q+1}^{(c)}}_{C^{N}_{x,t}}+ \norm{w_{q+1}^{(t)}}_{C^{N}_{x,t}}
 &\leq \frac 12 \lambda_{q+1}^{\sfrac{(3+5N)}{2}}
 \label{eq:warm:Merlot}
\end{align}
for $N \in \{ 0,1,2,3 \}$ and $p>1$, where $M$ is a universal constant. 
\end{proposition}
\begin{proof}[Proof of Proposition~\ref{prop:perturbation}]
For $i\geq 0$, from \eqref{e:a_est_L2} and \eqref{e:a_est_CN} we may estimate
\[
\norm{D^N a_{(\xi) }}_{L^2}
\les \delta_{q+1}^{\sfrac12} \ell^{-2N},
\]
where we have used that $\ell \delta_{q+1}^{-\sfrac 12} = \lambda_q^{-20 + \beta b} \les 1$, which follows from the restriction imposed on the smallness of $\beta b$. Since $\mathbb W_{(\xi)}$ is $(\sfrac{\T}{\lambda_{q+1}\sigma})^{3}$ periodic,  and the condition~\eqref{eq:ell:cond:1}  gives that  $\ell^{-2} \ll \lambda_{q+1}\sigma$, we may apply \eqref{e:W_Lp_bnd} with $N=K=0$, and Lemma~\ref{lem:Lp:independence} to conclude 
\begin{align}
\norm{a_{(\xi)} \mathbb  W_{(\xi)} }_{L^2}&
\lesssim \rho_i^{\sfrac 12} 2^{-i} \norm{ \mathbb W_{(\xi)}}_{L^2} 
\lesssim \rho_i^{\sfrac 12} 2^{-i} \,.\notag
\end{align}
Upon summing over $i \in \{0,\ldots,i_{\rm max}\}$, and appealing to \eqref{e:rho_sum}, we obtain \eqref{e:wp_est} for some fixed constant $M$ independent of any parameter. 

In order to bound the $L^2$ norm of $w_{q+1}^{(c)}$ we use \eqref{e:eta_Lp_bnd} and Lemma \ref{lem:many:bounds} to estimate
\begin{align}
\norm{\frac{1}{\lambda_{q+1}}\nabla \left(a_{(\xi)}\eta_{(\xi)} \right)\times W_{(\xi)}}_{L^2}
&\lesssim 
\frac{1}{\lambda_{q+1}}\left( \norm{\nabla a_{(\xi)}}_{L^{\infty}}\norm{\eta_{(\xi)}}_{L^2}+ \ \norm{ a_{(\xi)}}_{L^{\infty}}\norm{\nabla \eta_{(\xi)}}_{L^2}\right)\notag \\
&\lesssim 
\frac{ 1}{\lambda_{q+1}}\left(\ell^{-1}+\delta_{q+1}^{\sfrac 12} 2^i \lambda_{q+1} \sigma r\right) \notag\\
&\lesssim 
\delta_{q+1}^{\sfrac 12} 2^i \sigma r \, ,\notag
\end{align}
where we have used that $\ell^{-1}\leq \lambda_{q+1} \delta_{q+1}^{\sfrac 12} \sigma r$, which follows from \eqref{eq:r:sigma:def}--\eqref{eq:ell:def} since $b$ is sufficiently large. Analogously, bounding the summands in the definition of  $w_{q+1}^{(t)}$ we have
\begin{align*}
\norm{\frac{1}{2\mu}\mathbb P_{H}\mathbb P_{\neq 0}\left(a_{(\xi)}^2\eta_{(\xi)}^2 \xi\right)}_{L^2}
\les \frac{1}{\mu}
\norm{a}_{L^{\infty}}^2\norm{\eta_{(\xi)}}^2_{L^{4}} 
\les \frac{\delta_{q+1}4^i r^{\sfrac 32}}{\mu} \, .
\end{align*}
Summing in $i \in \{0,\ldots,i_{\rm max}\}$ and $\xi$,  and employing \eqref{eq:i:max:bound}, we obtain
\begin{align}\notag
\norm{w_{q+1}^{(c)}}_{L^2}+\norm{w_{q+1}^{(t)}}_{L^2}\les \frac{\delta_{q+1}^{\sfrac 12} \sigma r }{\ell^{\sfrac 12}} +\frac{\delta_{q+1} r^{\sfrac 32}}{\ell\mu}\les \frac{r^{\sfrac 32}}{\ell\mu} \delta_{q+1}^{\sfrac 12}\,.
\end{align}
In the above bound we have used the inequality $\ell^{\sfrac 12}\mu\leq \sigma^{-1} r^{\sfrac12}$, which follows from \eqref{eq:r:sigma:def}--\eqref{eq:ell:def}.

Now consider \eqref{e:wp_est_deriv}. Observe that by definition~\eqref{eq:w:q+1:p:def}, estimate \eqref{e:W_Lp_bnd}, and Lemma \ref{lem:many:bounds}, we have 
\begin{align}
\norm{w_{q+1}^{(p)}}_{W^{1,p}}
&\lesssim \sum_{i}\sum_{\xi\in \Lambda_{(i)}} \norm{a_{\xi}}_{C^1_{x,t}}\norm{\mathbb W_{(\xi)}}_{W^{1,p}}\notag\\
&\lesssim \sum_{i}\sum_{\xi\in \Lambda_{(i)}} \ell^{-1}\lambda_{q+1}r^{\sfrac 32-\sfrac 3p}
\notag\\
&\les \ell^{-2} \lambda_{q+1} r^{\sfrac 32 - \sfrac 3p} \, .
\label{eq:tilde:R:3}
\end{align}
Here we have also used \eqref{eq:i:max:bound} in order to sum over $i$. 
For the analogous bound on $w_{q+1}^{(c)}$, using \eqref{e:eta_Lp_bnd} and Lemma \ref{lem:many:bounds} we arrive at 
\begin{align*}
\norm{\frac{1}{\lambda_{q+1}}\nabla \left(a_{(\xi)}\eta_{(\xi)} \right)\times W_{(\xi)}}_{W^{1,p}}
&\lesssim 
\frac{1}{\lambda_{q+1}} \left( \norm{\nabla^2 \left(a_{(\xi)}\eta_{(\xi)} \right)}_{L^p}+\lambda_{q+1}\norm{\nabla \left(a_{(\xi)}\eta_{(\xi)} \right)}_{L^{p}}\right) \\
&\lesssim \frac{\norm{a_{(\xi)}}_{C^2}}{\lambda_{q+1}} \left( \norm{\eta_{(\xi)}}_{W^{2,p}}+\lambda_{q+1}\norm{\eta_{(\xi)} }_{W^{1,p}}\right)\\
&\lesssim \frac{\ell^{-2}}{\lambda_{q+1}} (\lambda_{q+1} \sigma r)^2 r^{\sfrac 32-\sfrac 3p}+\frac{\ell^{-2}}{\lambda_{q+1}} \lambda_{q+1}^2 \sigma  r^{\sfrac 52-\sfrac 3p}  \\
&\les \ell^{-2} \lambda_{q+1} r^{\sfrac 32-\sfrac 3p} \left(\sigma r\right) \,,
\end{align*}
where we used $\lambda_{q+1}\sigma r\leq \lambda_{q+1}$.
The above bound is consistent with \eqref{e:wp_est_deriv} for $w_{q+1}^{(c)}$ since summing over $i$ and $\xi$ loses an extra factor of $\ell^{-1}$ which may be absorbed since $\ell^{-1} \sigma r < 1$.
Similarly, in order to estimate $w_{q+1}^{(t)}$ we use bound~\eqref{e:eta_Lp_bnd} and Lemma~\ref{lem:many:bounds} to obtain
\begin{align*}
\norm{\frac{1}{\mu}\mathbb P_{H}\mathbb P_{\neq 0}\left(a_{(\xi)}^2\eta_{(\xi)}^2 \xi\right)}_{W^{1,p}}&\les \frac{1}{\mu}
\norm{a_{(\xi)}}_{C^1}\norm{a_{(\xi)}}_{L^{\infty}} \left( \norm{\nabla \eta_{(\xi)}}_{L^{{2p}}}  \norm{ \eta_{(\xi)}}_{L^{{2p}}} + \norm{\eta_{(\xi)}}^2_{L^{{2p}}}\right) \\
&\les \frac{1}{\mu} \ell^{-1}\delta_{q+1}^{\sfrac12} 2^i (\lambda_{q+1} \sigma r) r^{3-\sfrac3p}\,.
\end{align*}
Summing  in $i$ and $\xi$ and using \eqref{eq:i:max:bound} we obtain 
\begin{align*}
\norm{w_{q+1}^{(t)}}_{W^{1,p}} \les   \frac{1}{\mu} \ell^{-\sfrac 32}\delta_{q+1}^{\sfrac12} (\lambda_{q+1} \sigma r) r^{3-\sfrac3p} \les \ell^{-2} \lambda_{q+1} r^{\sfrac 32 - \sfrac 3p} \frac{\sigma r^{\sfrac 52}}{\mu}\,.
\end{align*}
Thus \eqref{e:wp_est_deriv} also holds for $w_{q+1}^{(t)}$, as a consequence of the inequality $\sigma r^{\sfrac 52}  \leq \mu$, which holds by \eqref{eq:r:sigma:def}.

Now consider the $L^p$ estimates of the time derivatives of $w_{q+1}^{(p)}$ and $w_{q+1}^{(c)}$. Estimates \eqref{e:W_Lp_bnd}  and \eqref{e:a_est_CN} yield
\begin{align}
\norm{\partial_t w_{q+1}^{(p)}}_{L^{p}} 
&\lesssim \sum_{i}\sum_{\xi\in \Lambda_{(i)}}\norm{a_{(\xi)}}_{C^{1}_{x,t}}\norm{\partial_t \mathbb W_{(\xi)}}_{L^{p}} \notag\\
&\lesssim \sum_{i}\sum_{\xi\in \Lambda_{(i)}}\ell^{-1} (\lambda_{q+1}\sigma r \mu) r^{\sfrac32-\sfrac 3p}
\notag\\
&\les  \ell^{-2}\lambda_{q+1}\sigma \mu r^{\sfrac52-\sfrac 3p} \, .
\notag
\end{align}
Similarly, using  \eqref{e:eta_Lp_bnd} and \eqref{e:a_est_CN}, we obtain
\begin{align*}
\norm{\partial_t w_{q+1}^{(c)}}_{L^p} &\les\sum_i \sum_{\xi\in \Lambda_{(i)}} \norm{\frac{1}{\lambda_{q+1}}\partial_t\left(\nabla \left(a_{(\xi)}\eta_{(\xi)} \right)\times W_{(\xi)}\right)}_{L^{p}}
\\
&\lesssim 
\frac{1}{\lambda_{q+1}} \sum_i \sum_{\xi\in \Lambda_{(i)}} \norm{\partial_t\nabla \left(a_{(\xi)}\eta_{(\xi)} \right)}_{L^{p}}\\
&\lesssim \frac{1}{\lambda_{q+1}} \sum_i \sum_{\xi\in \Lambda_{(i)}} \norm{a_{(\xi)}}_{C^2_{x,t}} \left(\norm{\partial_t \eta_{(\xi)}}_{W^{1,p}} + \norm{\eta_{(\xi)}}_{W^{1,p}} \right)\\
&\lesssim \ell^{-3} \lambda_{q+1} \sigma^2 \mu r^{\sfrac 72-\sfrac 3p} \, 
\end{align*}
a bound which is consistent with \eqref{e:wpc_temporal_est_deriv}, upon noting  that $\ell^{-1} \sigma r\leq 1$ holds.

For $N=0$, the bound \eqref{eq:warm:Merlot} holds for $w_{q+1}^{(p)}$ in view of \eqref{e:W_Lp_bnd}, \eqref{e:a_est_CN}, \eqref{eq:i:max:bound}, and the fact that $\ell^{-1} r^{\sfrac 32} \ll \lambda_{q+1}^{\sfrac 32}$. For the derivative bounds of $w_{q+1}^{(p)}$, we use \eqref{e:W_Lp_bnd} and \eqref{e:a_est_CN} to conclude
\begin{align}\notag
\norm{ a_{(\xi)} {\mathbb W}_{(\xi)}}_{C^N_{x,t}}
\les \norm{a_{(\xi)}}_{C^N_{x,t}}\norm{ {\mathbb W}_{(\xi)}}_{C_{x,t}^N}
\les \ell^{-N} (\lambda_{q+1} \sigma r \mu)^N r^{\sfrac32}
\end{align}
from which the first part of \eqref{eq:warm:Merlot} immediately follows in view of our parameter choices \eqref{eq:r:sigma:def}--\eqref{eq:ell:def}. Indeed, \eqref{eq:r:sigma:def} gives $\lambda_{q+1} \sigma r \mu = \lambda_{q+1}^{\sfrac{33}{16}} = \lambda_{q+1}^2 \lambda_{q+1}^{\sfrac{1}{16}}$ and $r^{\sfrac 32} = \lambda_{q+1}^{\sfrac 98}$. The bound for the $C_{x,t}^N$ norm of $w_{q+1}^{(c)}$ and $w_{q+1}^{(t)}$ follows mutatis mutandis.
\end{proof}

In view of the definitions of $w_{q+1}$ and $v_{q+1}$ in \eqref{eq:w:q+1:def} and \eqref{eq:v:q+1:def};  the estimates \eqref{eq:V_ell_est} and \eqref{eq:V_ell_est_N}; the identity $v_{q+1} - v_q = w_{q+1} + (v_\ell -v_q)$;  the bound $\ell \lambda_q^4 \delta_{q+1}^{- \sfrac 12}  + \ell^{-1}r^{\sfrac 32}\mu^{-1} \ll 1$, which holds since $b$ was taken to be sufficiently large; the estimates in Proposition~\ref{prop:perturbation} directly imply:
\begin{corollary}
\label{cor:perturbation}
\begin{align}
\norm{w_{q+1}}_{L^2}  &\leq \frac{3M}{4} \delta_{q+1}^{\sfrac12}\label{eq:w:q+1:L2}\\
\norm{v_{q+1} - v_q}_{L^2} &\leq M \delta_{q+1}^{\sfrac 12}\label{eq:increment:L2} \\
\norm{w_{q+1}}_{W^{1,p}}
&\les \ell^{-2} \lambda_{q+1} r^{\sfrac 32 - \sfrac 3p} \label{e:w_est_deriv}\\
\norm{w_{q+1}}_{C^{N}_{x,t}}  &\leq \frac 12 \lambda_{q+1}^{\sfrac{(3+5N)}{2}}\\
\norm{v_{q+1}}_{C^{N}_{x,t}}  &\leq  \lambda_{q+1}^{\sfrac{(3+5N)}{2}}
\end{align}
for  $N \in \{ 0, 1,2,3 \}$ and $p>1$.
\end{corollary}
Therefore, setting $N=1$ in the above estimate for $v_{q+1}$ we have proven that \eqref{e:V_ind} holds with $q$ replaced by $q+1$.  Also, \eqref{eq:increment:L2} proves the  velocity increment bound we have claimed in \eqref{e:interative_v}.

\section{Reynolds Stress}\label{sec:stress}
The main result of this section may be summarized as:
\begin{proposition}
\label{prop:stress:main}
There exists a $p>1$ sufficiently close to $1$ and an $\eps_R>0$ sufficiently small, depending only on $b$ and $\beta$ (in particular, independent of $q$), such that there exists a traceless symmetric $2$ tensor $\widetilde R$ and a scalar pressure field $\widetilde p$, defined implicitly in \eqref{eq:tilde:R:1} below, satisfying
\begin{align}\label{eq:tilde:R:def}
\partial_t v_{q+1}+\div\left(v_{q+1}\otimes v_{q+1}\right)+\nabla \widetilde p - {\nu \Delta v_{q+1}} =\div \widetilde R
\end{align}
and the bound 
\begin{align}
\norm{\widetilde R}_{L^p}\lesssim \lambda_{q+1}^{- 2 \eps_R}\delta_{q+2}
\label{eq:tilde:R:Lp}\,,
\end{align}
where the constant depends on the choice of $p$ and $\eps_R$.
\end{proposition}

An immediate consequence of Proposition~\ref{prop:stress:main} is that the desired inductive estimates \eqref{eq:R:q+1:ind}--\eqref{e:R_ind_C1} hold for a suitably defined Reynolds stress $\RR_{q+1}$ (see~\eqref{eq:Messi:10} below). We emphasize that compared to $\widetilde R$, the stress $\RR_{q+1}$ constructed below also obeys a satisfactory $C^1$ estimate.
\begin{corollary}
\label{cor:stress:main}
There exists a traceless symmetric $2$ tensor $\mathring R_{q+1}$ and a scalar pressure field $p_{q+1}$ such that
\begin{align}\notag
\partial_t v_{q+1}+\div\left(v_{q+1}\otimes v_{q+1}\right)+\nabla p_{q+1}-  {\nu \Delta v_{q+1}} =\div \mathring R_{q+1} \, .
\end{align}
Moreover, the following bounds hold
\begin{align}
\norm{\mathring R_{q+1}}_{L^1}&\leq \lambda_{q+1}^{- \eps_R}\delta_{q+2} \, ,
\label{e:R_final_L1}\\
\norm{\mathring R_{q+1}}_{C^{1}_{x,t}}&\leq \lambda_{q+1}^{10} \, .
\label{e:R_final_C1}
\end{align}
\end{corollary}
Before giving the proof of the corollary, we recall from~\cite[Definition 1.4]{BDLISJ15} the $2$-tensor valued elliptic operator ${\mathcal R}$ which has the property that  ${\mathcal R v}(x)$ is a symmetric trace-free matrix for each $x \in \T^3$, and ${\mathcal R}$ is an right inverse of the $\div$ operator, i.e.
\begin{align}\notag
\div {\mathcal R} v = v -  \fint_{\T^3} v(x) \, dx
\end{align}
for any smooth $v$. Moreover, we have the classical Calder\'on-Zygmund bound $\norm{|\nabla|{\mathcal R}}_{L^p\to L^p} \les 1$, and the Schauder estimates $\norm{\mathcal R}_{L^p\to L^p} +\norm{\mathcal R}_{C^0 \to C^0} \les 1$, for $p \in (1,\infty)$. Since throughout the proof the value of $p>1$ is independent of $q$, the implicit constants in these inequalities are uniformly bounded. 

\begin{proof}[Proof of Corollary~\ref{cor:stress:main}]
With $\widetilde R$ and $\widetilde p$ defined in Proposition~\ref{prop:stress:main}, we let 
\begin{align}
\label{eq:Messi:10}
\mathring R_{q+1}=\mathcal R( \Proj_H \div \widetilde R) \qquad \mbox{and} \qquad p_{q+1} = \tilde p - \Delta^{-1} \div \div \widetilde R\,.
\end{align}
With the parameter $p>1$ from Proposition~\ref{prop:stress:main}, using  that $\norm{{\mathcal R} \div }_{L^p\to L^p} \les 1$ we directly bound
\begin{align}\notag
\norm{\mathring R_{q+1}}_{L^1}&\lesssim \norm{\mathring R_{q+1}}_{L^p}\lesssim \norm{\widetilde R}_{L^p}\lesssim \lambda_{q+1}^{-2\eps_R}\delta_{q+2}\,.
\end{align}
The estimate \eqref{e:R_final_L1} then follows since the factor  $\lambda_{q+1}^{-\eps_R}$ can absorb any constant if we assume $a$ is sufficiently large. 

Now consider \eqref{e:R_final_C1}. Using equation~\eqref{eq:tilde:R:def} and the bounds of Corollary \ref{cor:perturbation} we obtain 
\begin{align}\notag
\norm{\mathring R_{q+1}}_{C^1}&= \norm{\mathcal R \Proj_H (\div \widetilde R)}_{C^1}\\
&\les \norm{\partial_t v_{q+1}+\div(v_{q+1}\otimes v_{q+1}) -  {\nu \Delta v_{q+1}}}_{C^1}\notag\\
&\lesssim \norm{\partial_t v_{q+1}}_{{C^1}}+\norm{v_{q+1}\otimes v_{q+1}}_{C^{2}} + {\norm{v_{q+1}}_{C^3}}\notag\\
&\lesssim \lambda_{q+1}^9\notag
\end{align}
by using the Schauder estimates $\norm{ \mathcal R \Proj_H}_{C^0 \to C^0} \les 1$.
Similarly, we have that
\begin{align}\notag
\norm{\partial_t \mathring R_{q+1}}_{L^{\infty}}
&\les \norm{\partial_t  (\partial_t v_{q+1}+\div(v_{q+1}\otimes v_{q+1})-  {\nu \Delta v_{q+1}})}_{C^0}\\
&\lesssim \norm{\partial_t^2 v_{q+1}}_{C^0}+\norm{\partial_t v_{q+1}\otimes v_{q+1}}_{C^{1}} +  {\norm{\partial_t v_{q+1}}_{C^2}}\notag\\
&\lesssim \lambda_{q+1}^9\notag
\end{align}
which concludes the proof of \eqref{e:R_final_C1} upon using the leftover power of $\lambda_{q+1}$ to absorb all $q$ independent constants.
\end{proof}

\subsection{Proof of Proposition~\ref{prop:stress:main}}
Recall that $v_{q+1} = w_{q+1} + v_{\ell}$, where $v_{\ell}$ is defined in Section~\ref{sec:mollify}. Using \eqref{e:NSE_reynolds_ell} and \eqref{eq:w:q+1:def} we obtain
\begin{align}
\div \tilde R - \nabla \tilde p
&=- \nu \Delta w_{q+1} +  \partial_t (w^{(p)}_{q+1} +w^{(c)}_{q+1} )  + \div( v_{\ell} \otimes w_{q+1} + w_{q+1} \otimes v_{\ell}) \notag\\
&\qquad + \div\left((w_{q+1}^{(c)}+w_{q+1}^{(t)}) \otimes w_{q+1}+ w_{q+1}^{(p)} \otimes (w_{q+1}^{(c)}+w_{q+1}^{(t)}) \right) \notag\\
&\qquad + \div(w_{q+1}^{(p)} \otimes w_{q+1}^{(p)} + \RR_{\ell})+\partial_t w^{(t)}_{q+1} \notag\\
&\qquad + \div(\tilde R_{\rm commutator}) - \nabla  p_\ell
\label{eq:tilde:R:1:*}\\
&=: \div \left( \tilde R_{\rm linear} + \tilde R_{\rm corrector} + \tilde R_{\rm oscillation} + \tilde R_{\rm commutator} \right) + \nabla \left(P - p_\ell\right).
\label{eq:tilde:R:1}
\end{align}
Here, the symmetric trace-free stresses $\tilde R_{\rm linear}$ and $\tilde R_{\rm corrector}$ are defined by applying the inverse divergence operator $\mathcal R$ to the first and respectively second lines of \eqref{eq:tilde:R:1:*}. The stress $\tilde R_{\rm commutator}$ was defined previously in \eqref{e:NSE_reynolds_ell:a}, while the stress $\tilde R_{\rm oscillation}$ is defined in Section~\ref{sec:oscillation} below. The pressure $P$ is given by \eqref{eq:referee:22} below.

Besides the already used inequalities between the parameters, $\ell$, $r$, $\sigma$, and $\lambda_{q+1}$, we shall use the following bound in order to achieve~\eqref{eq:tilde:R:Lp}:
\begin{align}
 \ell^{-2}\sigma \mu r^{\sfrac52-\sfrac 3p}+ (r^{\sfrac 32}\ell^{-1}\mu^{-1})^{\sfrac1p} \lambda_{q+1}^{3(1-\sfrac 1p)}+ \frac{ r^{3-\sfrac {3}{p}}}{\ell^3\lambda_{q+1}\sigma}
 +\frac{\sigma r^{4-\sfrac {3}{p}}}{\ell^3}
 + \lambda_q^{-10}  \les \lambda_{q+1}^{-2\eps_R} \delta_{q+2}
\label{eq:parameters}
\end{align}
In view of \eqref{eq:r:sigma:def}--\eqref{eq:ell:def}, the above inequality holds for $b$ sufficiently large, $\beta$ sufficiently small depending on $b$, parameters $\eps_R, p-1>0$ sufficiently small depending on $b$ and $\beta$, and for $\lambda_0 =  a$ sufficiently large depending on all these parameters and on $M$.

In view of the bound $\ell\lambda^6_q \les\lambda^{-10}_q$, the estimate \eqref{eq:Rc:bound} for $\tilde R_{\rm commutator}$ is consistent with \eqref{eq:tilde:R:Lp}. Hence it remains to consider the linear, corrector  and oscillation errors in \eqref{eq:tilde:R:1}.

\subsection{The linear and corrector errors}
In view of \eqref{eq:tilde:R:Lp}, we estimate contributions to the $\tilde R$ coming from the first line in \eqref{eq:tilde:R:1:*} as
\begin{align}
\norm{ \tilde R_{\rm linear}}_{L^p}  
&\les \norm{{\mathcal R} ( \nu \Delta w_{q+1} )}_{L^p} +\norm{\mathcal R(\partial_t (w^{(p)}_{q+1} +w^{(c)}_{q+1} ))}_{L^p}+ \norm{\mathcal R\div(v_\ell \otimes w_{q+1} + w_{q+1} \otimes v_\ell)}_{L^p}
\notag\\
& \les
\norm{w_{q+1}}_{W^{1,p}}+\frac{ {1} }{\lambda_{q+1}}\norm{\partial_t\mathcal R\curl\left(w^{(p)}_{q+1}\right)}_{L^p}+\norm{v_\ell}_{L^{\infty}}\norm{w_{q+1}}_{L^{p}} 
\notag\\
&\les
\lambda_q^4 \norm{w_{q+1}}_{W^{1,p}} +\frac{ {1} }{\lambda_{q+1}}\norm{\partial_t w^{(p)}_{q+1}}_{L^p}\notag\\
&\les \lambda_q^4 \ell^{-2} \lambda_{q+1} r^{\sfrac 32 - \sfrac 3p}+\frac{1}{\lambda_{q+1}}\ell^{-2}\lambda_{q+1}\sigma \mu r^{\sfrac52-\sfrac 3p}  \notag\\
&\les \ell^{-2} \sigma \mu r^{\sfrac52-\sfrac 3p}
\label{eq:tilde:R:2}
\end{align}
where we have used $\nu \leq 1$, $\lambda_q^4 \lambda_{q+1}\leq \sigma \mu r$,   the identity \eqref{e:curl_form}, the inductive estimate \eqref{e:V_ind} to bound  $\norm{v_\ell}_{L^\infty} \les \norm{v_\ell}_{C^1} \les \norm{v_q}_{C^1}$, estimates \eqref{e:wpc_temporal_est_deriv} and \eqref{e:w_est_deriv}. Next we turn to the  errors involving correctors, for which we appeal to $L^p$ interpolation, the Poincar\'e inequality, and Proposition~\ref{prop:perturbation}:
\begin{align*}
\norm{ \tilde R_{\rm corrector}}_{L^p} 
&\leq \norm{\mathcal R \div\left((w_{q+1}^{(c)}+w_{q+1}^{(t)}) \otimes w_{q+1}+ w_{q+1}^{(p)} \otimes (w_{q+1}^{(c)}+w_{q+1}^{(t)}) \right)}_{L^p}\\
&\lesssim   \norm{(w_{q+1}^{(c)}+w_{q+1}^{(t)}) \otimes w_{q+1}}_{L^1}^{\sfrac{1}{p}}   \norm{(w_{q+1}^{(c)}+w_{q+1}^{(t)}) \otimes w_{q+1}}_{L^\infty}^{1-\sfrac 1p}\\
&\quad+ \norm{w_{q+1}^{(p)} \otimes (w_{q+1}^{(c)}+w_{q+1}^{(t)})}_{L^1}^{\sfrac 1p}   \norm{w_{q+1}^{(p)} \otimes (w_{q+1}^{(c)}+w_{q+1}^{(t)})}_{L^\infty}^{1-\sfrac1p}
\\
&\lesssim (r^{\sfrac 32}\ell^{-1}\mu^{-1})^{\sfrac1p} \delta_{q+1}^{\sfrac1p} \lambda_{q+1}^{3(1-\sfrac1p)}
\\
&\lesssim (r^{\sfrac 32}\ell^{-1}\mu^{-1})^{\sfrac1p} \lambda_{q+1}^{3(1-\sfrac 1p)}\,.
\end{align*}
Due to \eqref{eq:parameters} this estimate is sufficient.

\subsection{Oscillation error}
\label{sec:oscillation}
In this section we estimate the remaining error, $\tilde R_{\rm oscillation}$ which obeys
\begin{align}
&\div \left( \tilde R_{\rm oscillation} \right) + \nabla P =\div\left(w^{(p)}_{q+1}\otimes w^{(p)}_{q+1} + \RR_{\ell}\right)+\partial_t w^{(t)}_{q+1}  \,,
\label{eq:referee:2}
\end{align}
where the pressure term $P$ is given by
\begin{align}
P =   \sum_{i\geq 0} \rho_i \chi_{(i)}^2  + \frac 12 \sum_{i,j}\sum_{\xi\in \Lambda_{(i)},\xi'\in \Lambda_{(j)}} a_{(\xi)}a_{(\xi')} \Proj_{\neq 0} \left(\WW_{(\xi)} \cdot \WW_{(\xi')} \right) - \sum_{i}\sum_{\xi\in \Lambda_{(i)}^+} \frac{1}{\mu} \Delta^{-1} \div \partial_t  \left(a_{(\xi)}^2  \eta_{(\xi)}^2 \xi\right)
\, .
\label{eq:referee:22}
\end{align}

Recall from the definition of $w_{q+1}^{(p)}$  and of the coefficients $a_{(\xi)}$, via \eqref{e:WW_id} we have
\begin{align}
&\div\left(w^{(p)}_{q+1}\otimes w^{(p)}_{q+1} \right) +   \div  {\RR_{\ell}}  \notag\\
&\qquad = \sum_{i,j}\sum_{\xi\in \Lambda_{(i)},\xi'\in \Lambda_{(j)}}\div\left(a_{(\xi)}a_{(\xi')}\mathbb W_{(\xi)}\otimes \mathbb W_{(\xi')}\right)+ \div \mathring R_{\ell} 
\notag \\
&\qquad = \sum_{i,j}\sum_{\xi\in \Lambda_{(i)},\xi'\in \Lambda_{(j)}}\div\left(a_{(\xi)}a_{(\xi')}\left(\mathbb W_{(\xi)}\otimes \mathbb W_{(\xi')}-\fint_{\T^3} \mathbb W_{(\xi)}\otimes \mathbb W_{(\xi')} \, dx \right)\right)  +\nabla \left( \sum_{i\geq 0} \rho_i \chi_{(i)}^2  \right)
\notag \\
&\qquad =\sum_{i,j}\sum_{\xi\in \Lambda_{(i)},\xi'\in \Lambda_{(j)}}\underbrace{\div\left(a_{(\xi)}a_{(\xi')}\mathbb P_{{\geq} \sfrac{\lambda_{q+1}\sigma}{2}}\left(\mathbb W_{(\xi)}\otimes \mathbb W_{(\xi')} \right)\right)}_{E_{(\xi,\xi')}} \; +\; \nabla \left( \sum_{i\geq 0} \rho_i \chi_{(i)}^2  \right) \, . \label{eq:referee}
\end{align}
Here we use that the minimal separation between active frequencies of ${\mathbb W}_{(\xi)} \otimes {\mathbb W}_{(\xi')}$ and the $0$ frequency is given by $\lambda_{q+1} \sigma$ for $\xi' = - \xi$, and by $c_\Lambda \lambda_{q+1}\geq \lambda_{q+1} \sigma$ for $\xi' \neq -\xi$. 
We proceed to estimate each symmetrized summand $E_{(\xi,\xi')}+E_{(\xi',\xi)}$ individually. We split
\begin{align*}
E_{(\xi,\xi')}+E_{(\xi',\xi)}&=  \mathbb P_{\neq 0}\left(\mathbb P_{{\geq}\sfrac{\lambda_{q+1}\sigma}{2}}\left(\mathbb W_{(\xi)}\otimes \mathbb W_{(\xi')}+\mathbb W_{(\xi')}\otimes \mathbb W_{(\xi)}\right)\nabla \left( a_{(\xi)}a_{(\xi')}\right) \right)\\
&\quad + \mathbb P_{\neq 0}\left(a_{(\xi)}a_{(\xi')}\div \left(\mathbb W_{(\xi)}\otimes \mathbb W_{(\xi')}+\mathbb W_{(\xi')}\otimes \mathbb W_{(\xi)}\right) \right)\notag\\
&=: E_{(\xi,\xi',1)}+E_{(\xi,\xi',2)} \, .\notag
\end{align*}
Here we used the fact that $E_{(\xi,\xi')}+E_{(\xi,\xi')}$ has zero mean to subtract the mean from each of the the two terms on the right hand side of the above. We have also removed the unnecessary frequency projection $P_{{\geq} \sfrac{\lambda_{q+1}\sigma}{2}}$ from the second term.

The term $E_{(\xi,\xi',1)}$ can easily be estimated using Lemma~\ref{lem:many:bounds} and Lemma~\ref{lem:comm:1}, estimate \eqref{eq:Merlot:2}, with $\lambda = \ell^{-1}$, $C_a = \ell^{-2}$, $\kappa = \sfrac{\lambda_{q+1} \sigma}{2}$, and $L$ sufficiently large,  as
\begin{align}
\norm{\mathcal R E_{(\xi,\xi',1)}}_{L^p}&\les \norm{\abs{\nabla}^{-1}  E_{(\xi,\xi',1)}}_{L^p}
\notag \\
& \les\norm{\abs{\nabla}^{-1}\Proj_{\neq 0} \left( \mathbb P_{{\geq} \sfrac{\lambda_{q+1}\sigma}{2}}\left(\mathbb W_{(\xi)}\otimes \mathbb W_{(\xi')}\right) \nabla \left( a_{(\xi)}a_{(\xi')}\right)  \right)}_{L^p}
\notag \\
&\les \frac{1}{\ell^{2}\lambda_{q+1}\sigma}\left(1+\frac{1}{\ell^L\left(\lambda_{q+1} \sigma\right)^{L-2}}\right)  \norm{\mathbb W_{(\xi)}\otimes \mathbb W_{(\xi')}}_{L^p} \notag\\
&\les   \frac{1}{\ell^{2}\lambda_{q+1}\sigma}  \norm{\mathbb W_{(\xi)}}_{L^{2p}} \norm{\mathbb W_{(\xi')}}_{L^{2p}} \notag\\
&\les \frac{ r^{3-\sfrac {3}{p}}}{\ell^{2}\lambda_{q+1}\sigma} \, .\notag
\end{align}
In the last inequality above we have used estimate~\eqref{e:W_Lp_bnd}, and in the second to last inequality we have used that for $b$ sufficiently large  and $L$ sufficiently large we have $\ell^{-L} (\lambda_{q+1} \sigma)^{2-L} \les 1$. Indeed, this inequality holds under the conditions $L\geq 3$ and $\sfrac{b(L-2)}{16} \geq 20 L$.
After summing in $i$ and $\xi$ we incur an additional loss of $\ell^{-1}$. By \eqref{eq:parameters} this bound is consistent with \eqref{eq:tilde:R:Lp}.

For the term $E_{(\xi,\xi',2)}$, we split into two cases: $\xi+\xi'\neq 0$ and $\xi+\xi'=0$. Let us first consider the case $\xi+\xi'\neq 0$. Applying the identity \eqref{eq:useful:1}, and using  \eqref{eq:frequency:support:2} we have
\begin{align*}
&a_{(\xi)}a_{(\xi')}\div\left( \mathbb  W_{(\xi)}\otimes \mathbb W_{( \xi')} + \mathbb  W_{(\xi')}\otimes \mathbb W_{( \xi)}\right)\notag\\
&\quad=a_{(\xi)}a_{(\xi')}\left(\left(W_{(\xi')}\cdot \nabla\left(\eta_{(\xi)}\eta_{(\xi')}\right)\right) W_{(\xi)}+\left(W_{(\xi)}\cdot \nabla\left(\eta_{(\xi)}\eta_{(\xi')}\right)\right) W_{(\xi')}\right)\notag\\
&\quad\quad+a_{(\xi)}a_{(\xi')}\eta_{(\xi)}\eta_{(\xi')}\nabla\left(W_{(\xi)}\cdot W_{(\xi')}\right)\notag\\
&\quad=a_{(\xi)}a_{(\xi')}\mathbb P_{\geq  c_{\Lambda}\lambda_{q+1}}\left(\nabla\left(\eta_{(\xi)}\eta_{(\xi')}\right)\left(W_{(\xi')} \otimes W_{(\xi)}+W_{(\xi)} \otimes W_{(\xi')}\right)\right)\notag\\
&\quad\quad+\nabla\left(a_{(\xi)}a_{(\xi')}\mathbb W_{(\xi)}\cdot\mathbb W_{(\xi')}\right)-\nabla \left(a_{(\xi)}a_{(\xi')}\right) \mathbb P_{\geq  c_{\Lambda} \lambda_{q+1}}\left(\mathbb W_{(\xi)}\cdot \mathbb W_{(\xi')}\right)\\
&\quad\quad-a_{(\xi)}a_{(\xi')}\mathbb P_{\geq  c_{\Lambda} \lambda_{q+1}}\left((W_{(\xi)}\cdot W_{(\xi')})\nabla (\eta_{(\xi)}\eta_{(\xi')})\right) \, .
\end{align*}
The second term is a pressure and to the remaining terms we apply the inverse divergence operator $\RSZ$. We estimate $\RSZ$ applied to the third term analogously to $E_{(\xi,\xi',1)}$, and $\RSZ$ applied to the fourth term can be estimated similarly to the first term. Thus it suffices to estimate $\RSZ $ applied to the first  term.  Applying \eqref{e:eta_Lp_bnd}, Lemma~\ref{lem:many:bounds}, estimate \eqref{eq:Merlot:2} of Lemma~\ref{lem:comm:1}, with $\lambda = \ell^{-1}$, $C_a = \ell^{-2}$,  $\kappa =c_{\Lambda}\lambda_{q+1}$, and for $b$ and $L$ sufficiently large ($L\geq 3$ and $b(L-2) \geq 20 L$ suffices), we obtain
\begin{align*}
&\norm{\mathcal R \left(a_{(\xi)}a_{(\xi')}\mathbb P_{\geq c_{\Lambda}\lambda_{q+1}}\left(\nabla\left(\eta_{(\xi)}\eta_{(\xi')}\right)(W_{(\xi')} \otimes W_{(\xi)}+W_{(\xi)} \otimes W_{(\xi')})\right)\right)}_{L^p}\\
&\quad\les \ell^{-2}  \left( 1+  \frac{1}{\ell^{L}\lambda_{q+1}^{L-2}}  \right)  \frac{\norm{\nabla\left(\eta_{(\xi)}\eta_{(\xi')}\right)}_{L^p}}{\lambda_{q+1}}\\
&\quad\les \frac{\sigma r^{4-\sfrac {3}{p}}}{\ell^2}\,.
\end{align*}
Summing in $\xi$ and $i$ we lose an additional $\ell^{-1}$ factor. By \eqref{eq:parameters} this bound is consistent with \eqref{eq:tilde:R:Lp}.

Now let us consider $E_{(\xi,\xi',2)}$ for the case $\xi+\xi'=0$. Applying the identity \eqref{eq:useful:2} we have
\begin{align}
E_{(\xi,-\xi,2)}&=\mathbb P_{\neq 0}\left(a_{(\xi)}^2\nabla \eta_{(\xi)}^2-a_{(\xi)}^2 \frac{\xi}{\mu}  \partial_t \left(\eta_{(\xi)}^2\right) \right)\notag \\
& = \nabla\left( a_{(\xi)}^2 \mathbb P_{{\geq}\sfrac{\lambda_{q+1}\sigma}{2}}\left(\eta_{(\xi)}^2\right)\right)- \mathbb P_{\neq 0}\left(\mathbb P_{{\geq}\sfrac{\lambda_{q+1}\sigma}{2}}\left(\eta_{(\xi)}^2\right)\nabla a_{(\xi)}^2\right) \notag\\
&\qquad- \frac{1}{\mu} \partial_t \mathbb P_{\neq 0}\left(a_{(\xi)}^2 \eta_{(\xi)}^2 \xi\right) + \frac{1}{\mu}  \mathbb P_{\neq 0}\left(\eta_{(\xi)}^2 \partial_t \left(a_{(\xi)}^2\right)\xi\right)\,.
\label{eq:Ronaldo:7}
\end{align}
Here we have used that $\Proj_{\neq 0} \eta_{(\xi)}^2 = \Proj_{\geq \sfrac{\lambda_{q+1}\sigma}{2}} \eta_{(\xi)}^2$, which holds since $\eta_{(\xi)}$ is $(\sfrac{\T}{\lambda_{q+1} \sigma})^3$-periodic.
Hence, summing in $\xi$ and $i$, using that $\eta_{(\xi)}=\eta_{(-\xi)}$, pairing with the $\partial_t w^{(t)}_{q+1}$ present in \eqref{eq:referee:2}, recalling the definition of $w^{(t)}_{q+1}$ in \eqref{e:temporal_corrector}, and noting that $\Id -\Proj_H = \nabla \Delta^{-1} \div$, we obtain
\begin{align}
&\sum_{i}\sum_{\xi\in \Lambda_{(i)}^+}E_{(\xi,-\xi,2)}+\partial_t w^{(t)}_{q+1}\notag \\
&\quad =
\nabla\left(  \sum_{i}\sum_{\xi\in \Lambda_{(i)}^+}a_{(\xi)}^2 \mathbb P_{{\geq}\sfrac{\lambda_{q+1}\sigma}{2}}\left(\eta_{(\xi)}^2\right)\right)
 -\sum_{i}\sum_{\xi\in \Lambda_{(i)}^+}\mathbb P_{\neq 0}\left(\mathbb P_{{\geq}\sfrac{\lambda_{q+1}\sigma}{2}}\left(\eta_{(\xi)}^2\right)\nabla a_{(\xi)}^2\right)
\notag \\
&\quad- \nabla \left( \sum_{i}\sum_{\xi\in \Lambda_{(i)}^+}\frac{1}{\mu} \Delta^{-1} \div \partial_t \left(a_{(\xi)}^2  \eta_{(\xi)}^2 \xi\right) \right)
 +\frac{1}{\mu}  \sum_{i}\sum_{\xi\in \Lambda_{(i)}^+} \Proj_{\neq 0}\left(\eta_{(\xi)}^2 \partial_t \left(a_{(\xi)}^2\right) \xi\right) \, .
\label{eq:peanuts}
\end{align}
The first and third terms are pressure terms  and to the remaining terms we apply the inverse divergence operator $\RSZ$. Thus it suffices to estimate $\RSZ$ applied to the second and the last term above. We split the second term of \eqref{eq:peanuts} into its summands, apply $\RSZ$, and estimate each term individually, similarly to the estimate of $\RSZ E_{(\xi,\xi',1)}$. Using \eqref{e:eta_Lp_bnd}, Lemma~\ref{lem:many:bounds} and Lemma~\ref{lem:comm:1}, estimate \eqref{eq:Merlot:2}, with $\lambda = \ell^{-1}$, $C_a = \ell^{-2}$, $\kappa = \sfrac{\lambda_{q+1} \sigma}{2}$, and for $b$ and $L$ sufficiently large ($L\geq 3$ and $\sfrac{b(L-2)}{16} \geq 20 L$),  we obtain
\begin{align}
\norm{\mathcal R\left(\mathbb P_{{\geq}\sfrac{\lambda_{q+1}\sigma}{2}}(\eta_{(\xi)}^2)\nabla a_{(\xi)}^2\right)}_{L^p}&\les
\frac{ 1}{\ell^2\lambda_{q+1}\sigma}\left(1+\frac{1}{\ell^L\left(\lambda_{q+1} \sigma\right)^{L-2}}\right)  \norm{\eta_{(\xi)}^2}_{L^p}\les \frac{ r^{3-\sfrac {3}{p}}}{\ell^2\lambda_{q+1}\sigma}
\,.
\label{e:sutter_home1}
\end{align}
Summing over $\xi$ and $i$ we lose a factor of $\ell^{-1}$.
Lastly, applying $\RSZ$ to the last term on the right side of \eqref{eq:peanuts}, and the bound on each summand is a simple consequence of \eqref{e:eta_Lp_bnd} and Lemma~\ref{lem:many:bounds}:
\begin{align}
\norm{\mathcal R\left(\frac{1}{\mu}\sum_{i}\sum_{\xi\in \Lambda_{(i)}^+} \mathbb P_{\neq 0}\left(\partial_t (a_{(\xi)}^2)\eta_{(\xi)}^2 \xi\right)\right)}_{L^p}&\les \frac{1}{\mu}\sum_{i}\sum_{\xi\in \Lambda_{(i)}^+}\norm{\partial_t (a_{(\xi)}^2)\eta_{(\xi)}^2 \xi}_{L^p}\notag\\
&\les \frac{1}{\mu}\sum_{i}\sum_{\xi\in \Lambda_{(i)}^+}\norm{a_{(\xi)}}_{C^{1}_{t,x}}\norm{a_{(\xi)}}_{L^{\infty}}\norm{\eta_{(\xi)}}_{L^{2p}}^2\notag\\
&\les \frac{1}{\mu}\sum_{i}\ell^{-1}\delta_{q+1}^{\sfrac12}2^ir^{3-\sfrac{3}{p}}\notag\\
&\les \frac{\ell^{-2}r^{3-\sfrac{3}{p}}}{\mu} \les \frac{r^{3-\sfrac{3}{p}}}{\ell^2\lambda_{q+1}\sigma}\label{e:sutter_home2}
\end{align}
where we have used that $\lambda_{q+1}\sigma\leq \mu$. 
Using \eqref{eq:parameters}, the bounds \eqref{e:sutter_home1} and \eqref{e:sutter_home2} and consistent with  \eqref{eq:tilde:R:Lp}, which concludes the proof of Proposition~\ref{prop:stress:main}.

\section{The energy iterate}\label{s:energy}
\begin{lemma}\label{l:energy_useful}
For all $t$ and $t'$ satisfying $\abs{t-t'} \leq 2\ell$, and all $i\geq 0$, we have
\begin{align}
\abs{e(t')-e(t'')}
&\les\ell^{\sfrac12} \label{e:e_diff}\\
\abs{\int_{\mathbb T^3}\abs{v_q(x,t)}^2\,dx-\int_{\mathbb T^3}\abs{v_q(x,t')}^2\,dx}&\les \ell^{\sfrac12}\label{e:v_diff}\\
\abs{\int_{\mathbb T^3}\left(\chi_i^2(x,t)-\chi_i^2(x,t')\right)\,dx}&\les \ell^{\sfrac12}\label{e:chi_diff}
\\
\abs{\rho(t)-\rho(t')}
&\les\ell^{\sfrac12}\label{e:rho_diff}
\end{align}
\end{lemma}
\begin{proof}[Proof of Lemma~\ref{l:energy_useful}]
In the proof of the lemma, we crudely use a factor of $\lambda_q$ to absorb constants. First note that \eqref{e:e_diff} follows immediately from the assumed estimate $\norm{e}_{C^1_t}\leq M_e$. Using \eqref{e:V_ind} we have
\begin{align}
\abs{\int_{\mathbb T^3}\abs{v_q(x,t)}^2\,dx-\int_{\mathbb T^3}\abs{v_q(x,t')}^2\,dx}&\les \ell \norm{v_q}_{C^1_{t,x}}^2\les \lambda_{q}^{8}\ell\,,\notag
\end{align}
which implies \eqref{e:v_diff}. The estimate  \eqref{e:chi_diff} follows in a similar fashion, from  Lemma~\ref{lem:psi:bounds}. Finally, \eqref{e:rho_diff} follows directly the definition of $\rho(t)$, \eqref{e:chi_lower_bnd} and the bounds \eqref{e:e_diff}--\eqref{e:chi_diff}  above.
\end{proof}

\begin{lemma}\label{l:inductive_energy_step_1}
 If $\rho_0(t)\neq 0$ then the energy of $v_{q+1}$ satisfies the following estimate:
\begin{equation}\label{e:inductive_energy_step_1}
    \abs{e(t)-\int_{\T^3}\abs{v_{q+1}(x,t)}^2\,dx-\frac{\delta_{q+2}}{2} }\leq \frac{\delta_{q+2}}{4}\,.
\end{equation}
\end{lemma}
Note, the above lemma implies that if $\rho_0(t)\neq 0$, then $e(t) - \int_{\mathbb T^3}\abs{ v_q(x,t)}^2~dx >   \frac{\delta_{q+1}}{100}$, and thus \eqref{eq:zero_reynolds} is an empty statement for such times.
\begin{proof}[Proof of Lemma~\ref{l:inductive_energy_step_1}]
By definition we have
\begin{equation}\label{e:dornfelder}
\int_{\T^3}\abs{v_{q+1}(x,t)}^2\,dx=\int_{\T^3}\abs{v_{\ell}(x,t)}^2\,dx + 2\int_{\T^3} w_{q+1}(x,t)\cdot v_{\ell}(x,t)\,dx+\int_{\T^3}\abs{w_{{q+1}}(x,t)}^2\,dx\,.
\end{equation}
Using \eqref{e:WW_id}, similarly to \eqref{e:WW_id_almost}, we have that
\begin{align}
&\int_{\T^3} \abs{w_{q+1}^{(p)}(x,t)}^2dx - 3 \sum_{i\geq 0} \rho_i \int_{\mathbb T^3}\chi_{(i)}^2(x,t)\,dx
\notag\\
&=  \sum_{i,j\geq 0} \sum_{\xi \in \Lambda_{(i)},\xi' \in \Lambda_{(j)}}\int_{\T^3} a_{(\xi)} a_{(\xi')} \Proj_{\geq \sfrac{\lambda_{q+1} \sigma}{2}}(\mathbb W_{(\xi)}\cdot \mathbb W_{(-\xi')}) dx 
=: E_{\rho}(t).\notag
\end{align}
Using the standard integration by parts argument 
$|\int_{\T^3} f \Proj_{\geq \mu} g dx| = |\int_{\T^3}\abs{\nabla}^L f \abs{\nabla}^{-L}\Proj_{\geq \mu} g dx| \les \norm{g}_{L^2}\mu^{-L}  \norm{f}_{C^{L}}$,
with $L$ sufficiently large, we obtain from \eqref{e:a_est_CN}, since $\ell^{-1} \ll \lambda_{q+1} \sigma$, that 
\begin{align}\label{e:Freixenet}
\left|\int_{\T^3} \abs{w_{q+1}^{(p)}(x,t)}^2 dx - 3 \sum_{i\geq 0} \rho_i \int_{\mathbb T^3}\chi_{(i)}^2(x,t)\,dx \right| = \abs{E_\rho(t)} \leq  \ell^{\sfrac 12}.
\end{align}

We consider two sub-cases:  $\rho(t)\neq 0$ and  $\rho(t)= 0$. First consider the case $\rho(t)\neq 0$, then using the definition of $\rho$ we obtain
\begin{align*}
&3 \sum_{i\geq 0} \rho_i \int_{\mathbb T^3}\chi_{(i)}^2(x,t)\,dx\\
&=3 \rho(t)\int_{\mathbb T^3}\chi_{(0)}^2(x,t)\,dx+ 3 \left(\rho_0(t)-\rho(t)\right)\int_{\mathbb T^3}\chi_{(0)}^2(x,t)\,dx +3 \sum_{i\geq 1} \rho_i \int_{\mathbb T^3}\chi_{(i)}^2(x,t)\,dx \\
&= e(t)-\int_{\mathbb T^3}\abs{v_q(x,t)}^2\,dx+3 \left(\rho_0(t)-\rho(t)\right)\int_{\mathbb T^3}\chi_0^2(x,t)\,dx -\frac{\delta_{q+2}}{2}\,.
\end{align*}
For the case that $\rho(t)=0$ we have that by continuity for some $t'\in (t-\ell, t+\ell)$
\begin{align}\notag
e(t')-\int_{\mathbb T^3}\abs{v_q(x,t')}^2\,dx- 3 \sum_{i\geq 1}\rho_i\int_{\mathbb T^3}\chi_{i}^2(x,t')\, dx-\frac{\delta_{q+2}}{2}=0\,.
\end{align}
Thus applying Lemma \ref{l:energy_useful} we conclude that for either case $\rho(t)\neq 0$ or $\rho(t)=0$ \begin{align}\label{e:malbec}
\abs{3 \sum_{i\geq 0} \rho_i \int_{\mathbb T^3}\chi_{(i)}^2(x,t)\,dx-e(t)+\int_{\mathbb T^3}\abs{v_q(x,t)}^2\,dx +\frac{\delta_{q+2}}{2}}\les \ell^{\sfrac12}\,.
\end{align}
When $\rho(t) \neq 0$ in the above estimate we have used the bound $|\rho_0(t) - \rho(t)| \les \ell^{\sfrac 12}$, which follows from the definition of $\rho_0$ in \eqref{eq:rho:0:def}, and the estimate \eqref{e:rho_diff} established earlier.

Observe that using \eqref{e:V_ind}, the definition of $v_\ell$, and \eqref{eq:V_ell_est} we have
\begin{align}
\left| \int_{\T^3} |v_q(x,t)|^2 dx - \int_{\T^3} |v_\ell(x,t)|^2 dx \right| \les \norm{v_q}_{L^\infty} \norm{v_\ell - v_q}_{L^\infty} \les \lambda_q^8 \ell \les \ell^{\sfrac 12}.
\label{eq:energy:ell:q}
\end{align}
Further, using also \eqref{eq:frequency:support}, \eqref{eq:V_ell_est_N}, \eqref{e:a_est_CN}, and integration by parts, we obtain
\begin{align*}
\abs{\int_{\T^3} w_{q+1}(x,t)\cdot v_{\ell}(x,t)}\,dx
&\leq \norm{w_{q+1}^{(c)} + w_{q+1}^{(t)}}_{L^2} \norm{v_{\ell}}_{C^1} +  \sum_{i\geq 0}\sum_{\xi\in \Lambda_{(i)}}\abs{\int_{\T^3} a_{(\xi)} \mathbb W_{(\xi)}(x,t)\cdot v_{\ell}(x,t)\,dx} \\
&\les\ell^{-2} r^{3/2} \mu^{-1} +   \sum_{i\geq 0}\sum_{\xi\in \Lambda_{(i)}}\lambda_{q+1}^{-N}\norm{ a_{(\xi)}v_{\ell}}_{C^N}\\
&\les \ell^{-2} \lambda_{q+1}^{-\sfrac 18} +  \lambda_{q+1}^{-N}\ell^{-N-2} \, .
\end{align*}
In the last line we have used \eqref{eq:r:sigma:def} and the fact that summing over $\xi$ and $i$ costs at most an extra $\ell^{-1}$. Taking $N$ sufficiently large we obtain
\begin{align}\label{e:yellow_tail}
\abs{\int_{\T^3} w_{q+1}(x,t)\cdot v_{\ell}(x,t)}\,dx&\les \ell^{-2} \lambda_{q+1}^{-\sfrac 18}  \les \ell^{\sfrac12}\, ,
\end{align}
upon taking $b \geq 400$.
Using \eqref{e:wp_est} and \eqref{e:wc_est} yields
\begin{align}\label{e:chianti}
\abs{\int_{\T^3} |w_{q+1}(x,t)|^2 dx - \int_{\T^3} \abs{w^{(p)}_{q+1}(x,t)}^2\,dx}
&\les 
\left( \norm{w_{q+1}}_{L^2} + \norm{w_{q+1}^{(p)}}_{L^2}\right) \left(\norm{w_{q+1}^{(c)}}_{L^2} + \norm{w_{q+1}^{(t)}}_{L^2} \right)
\notag\\
&\lesssim \delta_{q+1} \ell^{-1}\mu^{-1} \notag\\
&\les \ell^{\sfrac12}\,.
\end{align}
Thus we conclude from \eqref{e:dornfelder}, \eqref{e:Freixenet}, \eqref{e:malbec}, \eqref{e:yellow_tail} and \eqref{e:chianti}
\begin{align}\notag
 \abs{e(t)-\int_{\T^3}\abs{v_{q+1}(x,t)}^2\,dx-\frac{\delta_{q+2}}{2} }\les \ell^{\sfrac12}
\end{align}
from which \eqref{e:inductive_energy_step_1} immediately follows.
\end{proof}

\begin{lemma}\label{l:inductive_energy_step_2}
If $\rho_0(t)=0$ then $v_{q+1}(\cdot,t) \equiv 0$,  $\mathring R_{q+1}(\cdot,t)\equiv 0$ and
\begin{align}\label{e:inductive_energy_step_2}
e(t)-\int_{\mathbb T}\abs{v_{q+1}(x,t)}^2\leq \frac{3\delta_{q+2}}{4}\,.
\end{align}
\end{lemma}
\begin{proof}[Proof of Lemma~\ref{l:inductive_energy_step_2}]
Since $\rho_0(t)=0$, it follows from the definition of $\rho_0$ and $\rho$ that for all $t'\in(t- \ell,t+ \ell)$ we have
\begin{align}\notag
e(t')-\int_{\mathbb T^3}\abs{v_q(x,t')}^2\,dx-3 \sum_{i\geq 1}\rho_i\int_{\mathbb T^3}\chi_{i}^2(x,t')\, dx\leq \frac{\delta_{q+2}}{2}\,.
\end{align}
Using \eqref{e:rho_i_contrib}, this implies that
\begin{align}\label{e:energy_silly}
e(t')-\int_{\mathbb T^3}\abs{v_q(x,t')}^2\,dx-\frac{\delta_{q+2}}{2}\les \lambda_{q}^{-\eps_R}\delta_{q+1}\,.
\end{align}
Using that $\lambda_{q}^{-\eps_R}$ and the ratio $\delta_{q+2}\delta_{q+1}^{-1}$ can absorb any constant, from \eqref{eq:zero_reynolds} we conclude that $v_q(\cdot,t')\equiv 0$ and  $\mathring R_q(\cdot,t')\equiv 0$ for all $t' \in (t- \ell, t+ \ell)$. Hence $v_\ell(\cdot,t) \equiv 0$ and $\RR_{\ell}(\cdot,t)\equiv 0$.  This in turn implies that $ \chi_i(\cdot,t) \equiv 0$ for all $i\geq 1$. Since in addition $\rho_0(t)=0$, it follows by \eqref{eq:a:oxi:def} $a_{(\xi)}(\cdot,t) = 0$ for all $i\geq 0$, and thus that $w_{q+1}(\cdot,t)\equiv 0$. Hence we have that $v_{q+1}(\cdot,t)\equiv 0$. Moreover, since the $\{\chi_i (x,t)\}_{i\geq 1}$ and $\rho_0^{\sfrac 12}(t)$ are non-negative smooth functions, it follows that $\partial_t \chi_i (\cdot,t)\equiv 0$ for all $i\geq 1$ and that $\partial_t \rho_0^{\sfrac 12}(t) = 0$. Hence we also obtain $\partial_t a_{(\xi)}(\cdot ,t) \equiv 0$, and from the definition of $w_{q+1}$ we have $\partial_t w_{q+1}(\cdot,t)\equiv 0$. Since $v_q$ vanishes on $(t-\ell,t+\ell)$ and  $v_\ell$, $w_{q+1}$, $\partial_tw_{q+1}$, $\RR_{\ell}$, all vanish at time $t$, it follows from \eqref{e:NSE_reynolds_ell:a} and \eqref{eq:tilde:R:1} that $\tilde R(\cdot,t)\equiv 0$, and therefore $\RR_{q+1}(\cdot,t)\equiv 0$.

Using \eqref{e:energy_silly} and \eqref{eq:energy:ell:q} (note that $\ell^{-\sfrac 14}$ may be used to absorb constants) we obtain
\begin{align*}
e(t)-\int_{\mathbb T^3}\abs{v_{q+1}(x,t)}^2
&=e(t)-\int_{{\mathbb T}^3}\abs{v_{q}(x,t)}^2 + \int_{{\mathbb T}^3}\abs{v_{q}(x,t)}^2 - \int_{{\mathbb T}^3}\abs{v_{\ell}(x,t)}^2 \notag\\
&\leq \frac{5\delta_{q+2}}{8} + \ell^{\sfrac 14} \leq \frac{3 \delta_{q+2}}{4}\,.
\end{align*}
In the last inequality we used that $\beta b^2$ is sufficiently small. 
Hence we obtain \eqref{e:inductive_energy_step_2}.
\end{proof}
 
We conclude this section by using Lemmas \ref{l:inductive_energy_step_1} and \ref{l:inductive_energy_step_2} to conclude 
\eqref{eq:energy_ind} and \eqref{eq:zero_reynolds} for $q+1$. Observe that the estimates \eqref{e:inductive_energy_step_1} and  \eqref{e:inductive_energy_step_2}, together imply \eqref{eq:energy_ind} for $q+1$. From \eqref{e:inductive_energy_step_1}, if
\begin{align}\notag
e(t) - \int_{\mathbb T^3}\abs{ v_{q+1}(x,t)}^2~dx\leq \frac{\delta_{q+2}}{100}
\end{align}
then $\rho_0(t)=0$. Hence from  Lemma \ref{l:inductive_energy_step_2} we obtain that  $v_{q+1} \equiv 0$ and $\mathring R_{q+1}\equiv 0$, from which we conclude \eqref{eq:zero_reynolds}. 
\appendix

\section{$L^p$ product estimate}
\label{app:Lp:product}

\begin{proof}[Proof of Lemma~\ref{lem:Lp:independence}]
For convenience we give here the proof from~\cite{BMV17}. We first consider the case $p=1$.
With these assumptions we have
\begin{equation}\notag
\norm{f g}_{L^1} \leq \sum_j \int_{T_j} \abs{fg}
\end{equation}
where $T_j$ are cubes of side-length $\frac{2\pi}{\kappa}$. For any function $h$, let $\overline h_j$ denote its mean on the cube $T_j$.  Observe that for $x \in T_j$ we have
\begin{align*}
\abs{f(x)}=\abs{\overline f_j+f(x)-\overline f_j }
&\leq \abs{\overline f_j}+\sup_{T_j}\abs{f(x) -\overline f_j }\\
&\leq\abs{\overline f_j}+\frac{2\pi\sqrt{3}}{\kappa}\sup_{T_j}\abs{Df}\\
&\leq\abs{\overline f_j}+\frac{2\pi\sqrt{3}}{\kappa}\abs{\overline {Df_j}}+\frac{2\pi\sqrt{3}}{\kappa}\sup_{T_j}\abs{Df-\overline {Df}_j}\\
&\leq\abs{\overline f_j}+\frac{2\pi\sqrt{3}}{\kappa}\abs{\overline{ Df_j}}+\frac{6\pi}{\kappa^2}\sup_{T_j}\abs{D^2f}\\
&\leq\abs{\overline f_j}+\frac{2\pi\sqrt{3}}{\kappa}\abs{\overline {Df_j}}+\frac{6\pi}{\kappa^2}\sup_{T_j}\abs{\overline {D^2f}_j}
+\frac{6\pi}{\kappa^2}\sup_{T_j}\abs{D^2f-\overline {D^2f}_j}.
\end{align*}
Iterating this procedure $M$ times we see that on $T_j$ we have the pointwise estimate
\begin{align}\notag
\abs{f}
&\leq \sum_{m=0}^M (2\pi\sqrt{3} \kappa^{-1})^m \abs{\overline {D^m f}_j} + (2\pi\sqrt{3} \kappa^{-1})^M \|D^M f\|_{L^\infty}.
\end{align}
Upon multiplying the above by $|g|$ and integrating over $T_j$, and then summing over $j$, we obtain
\begin{align*}
\norm{f g}_{L^1(\T^3)} 
&\leq \sum_{j} \int_{T_j} \left( |g|  \sum_{m=0}^M (2\pi\sqrt{3} \kappa^{-1})^m \abs{\overline {D^m f}_j} \right) dx + (2\pi\sqrt{3} \kappa^{-1})^M \|D^M f\|_{L^\infty} \|g\|_{L^1} \notag\\
&\leq \sum_{m=0}^M (2\pi\sqrt{3} \kappa^{-1})^{m} \left( \sum_j \frac{1}{|T_j|} \norm{D^m f}_{L^{1}(T_j)}\norm{g}_{L^1(T_j)} \right)
+ (2\pi\sqrt{3} \kappa^{-1})^M \|D^M f\|_{L^\infty} \|g\|_{L^1}.
\end{align*}
Since $g$ is a $T_j$-periodic function, we have
\begin{align*}
\norm{g}_{L^1(\T^3)} = \frac{|\T^3|}{|T_j|} \norm{g}_{L^1(T_j)} 
\end{align*}
for any value of $j$, and since the interiors of the $\{T_j\}$ are mutually disjoint, based on the assumption on the $L^1$ cost of a derivative acting on $f$ and the Sobolev embedding, we conclude from the above that (here we used the Sobolev embedding of $W^{d+1,1} \subset L^\infty$)
\begin{align}
\norm{f g}_{L^1(\T^3)} 
&\leq \frac{1}{|\T^3|} \norm{g}_{L^1(\T^3)} \sum_{m=0}^M (2\pi\sqrt{3} \kappa^{-1})^{m}  \norm{D^m f}_{L^{1}(\T^3)} 
+ (2\pi\sqrt{3} \kappa^{-1})^M \|D^{M+4} f\|_{L^1} \|g\|_{L^1} \notag\\
&\leq \frac{1}{|\T^3|} \norm{g}_{L^1(\T^3)} \sum_{m=0}^M (2\pi\sqrt{3} \kappa^{-1})^{m}   \lambda^m C_f 
+ (2\pi\sqrt{3} \kappa^{-1})^M \lambda^{M+4} C_f \|g\|_{L^1} \notag\\
&\leq (1 + 2 |\T^3|) C_f \|g\|_{L^1(\T^3)}.\notag
\end{align}
The case $p=2$, follows from the case $p=1$ applied to the functions $f^2$ and $g^2$, and from the bound
\begin{align}
\|D^m (f^2)\|_{L^1} \leq \sum_{k=0}^m {m\choose k} \|D^k f\|_{L^2} \|D^{m-k} f\|_{L^2} \leq \sum_{k=0}^m {m\choose k} \lambda^mC_f^2 = (2\lambda)^mC_f^2.\notag
\end{align}
Here we are thus using that $4\pi \sqrt{3} \lambda \kappa^{-1} \leq 2/3 < 1$ so that we have a geometric sum.
\end{proof}

\section{Commutator estimate}
\label{app:commutator}
\begin{lemma}
\label{lem:comm:1}
Fix $\kappa \geq 1$, $p \in (1,2]$, and a sufficiently large $L \in \N$. Let $a \in C^L(\T^3)$ be such that there exists $1 \leq \lambda \leq \kappa$, and $C_a >0$ with
\begin{align}
\norm{D^j a}_{L^\infty} \leq C_a \lambda^j
\label{eq:Merlot:1}
\end{align}
for all $0 \leq j \leq L$. Assume furthermore that $\int_{\T^3} a(x) \Proj_{\geq \kappa} f(x) dx = 0$.
Then we have
\begin{align}
\norm{ |\nabla|^{-1} (a \; \Proj_{\geq \kappa} f)}_{L^p} 
\les C_a  \left( 1+  \frac{\lambda^L}{\kappa^{L-2}}  \right)  \frac{\norm{f}_{L^p}}{\kappa}
\label{eq:Merlot:2}
\end{align}
for any $f\in L^p(\T^3)$, where the implicit constant depend on $p$ and $L$.
\end{lemma}
\begin{proof}[Proof of Lemma~\ref{lem:comm:1}]
We have that  
\begin{align}
|\nabla|^{-1} (a \; \Proj_{\geq \kappa} f) 
&= |\nabla|^{-1} (\Proj_{\leq \sfrac{\kappa}{2}} a \; \Proj_{\geq \kappa} f)  + |\nabla|^{-1} (\Proj_{\geq \sfrac{\kappa}{2}}  a \; \Proj_{\geq \kappa} f) \notag\\
&=(\Proj_{\geq \sfrac{\kappa}{2}} |\nabla|^{-1}) (\Proj_{\leq \sfrac{\kappa}{2}} a \; \Proj_{\geq \kappa} f)  + |\nabla|^{-1} (\Proj_{\geq \sfrac{\kappa}{2}}  a \; \Proj_{\geq \kappa} f)  \, .
\label{eq:Merlot:4}
\end{align}
Note that $\int_{\T^3} \Proj_{\geq \sfrac{\kappa}{2}} g(x) dx = 0$ for any function $g$, and thus the assumption that $a \, \Proj_{\geq \kappa} f$ has zero mean on $\T^3$, implies that $\Proj_{\geq \sfrac{\kappa}{2}} a \, \Proj_{\geq \kappa} f$ also has zero mean on $\T^3$. We then use 
\begin{align}\notag
\norm{|\nabla|^{-1} \Proj_{\geq \sfrac{\kappa}{2}} }_{L^p\to L^p} \les \frac{1}{\kappa}
\end{align}
which is a direct consequence of the Littlewood-Paley decomposition, and the bound
\begin{align}\notag
\norm{|\nabla|^{-1} \Proj_{\neq 0}}_{L^p \to L^p} \les 1
\end{align}
which is a direct consequence of Schauder estimates (see~\cite{GrafakosClassical}). 
Combining these facts and appealing to the embedding $W^{1,4}(\T^3) \subset L^\infty(\T^3)$, we obtain
\begin{align}
\norm{|\nabla|^{-1} (a \; \Proj_{\geq \kappa} f)}_{L^p}
&\les \frac{1}{\kappa} \norm{\Proj_{\leq \sfrac{\kappa}{2}} a \; \Proj_{\geq \kappa} f}_{L^p}
+ \norm{\Proj_{\geq \sfrac{\kappa}{2}}  a \; \Proj_{\geq \kappa} f}_{L^p} \notag\\
&\les \left( \norm{a}_{L^\infty}  + \kappa \norm{D \Proj_{\geq \sfrac{\kappa}{2}}a }_{L^4} \right)\frac{ \norm{f}_{L^p}}{\kappa} \notag\\
&\les \left( \norm{a}_{L^\infty}  + \kappa^{2-L} \norm{D^L \Proj_{\geq \sfrac{\kappa}{2}}a }_{L^4} \right)\frac{ \norm{f}_{L^p}}{\kappa} \notag\\
&\les \left( \norm{a}_{L^\infty}  + \kappa^2  \frac{\norm{D^L a }_{L^\infty}}{\kappa^L} \right)\frac{ \norm{f}_{L^p}}{\kappa}\notag
\end{align}
and the proof of \eqref{eq:Merlot:2} is concluded in view of assumption \eqref{eq:Merlot:1}.
\end{proof}

\section*{Acknowledgments} 
The work of T.B. has been partially supported by the National Science Foundation grant DMS-1600868.
V.V. was partially supported by the National Science Foundation grant DMS-1652134 and by an
Alfred P. Sloan Research Fellowship. The authors would like to thank Maria Colombo, Camillo De Lellis, Alexandru Ionescu, Igor Kukavica, and Nader Masmoudi for their valuable
suggestions and comments.


\end{document}